\documentclass[a4paper,12pt,oneside]{amsart} 
\usepackage{amsmath,amsfonts,amssymb,amsthm}
\usepackage{epsfig}
\usepackage{enumerate}

\theoremstyle{plain}
\newtheorem{theorem}{Theorem}[section]

\newtheorem{lemma}[theorem]{Lemma}
\newtheorem{proposition}[theorem]{Proposition}

\theoremstyle{definition}
\newtheorem{definition}[theorem]{Definition}

\newtheorem{example}[theorem]{Example}
\newtheorem{remark}[theorem]{Remark}

\numberwithin{equation}{section}

\newcommand{\R}{{\mathbb R}}
\newcommand{\N}{{\mathbb N}}

\providecommand{\vint}[1]{\mathchoice
          {\mathop{\vrule width 5pt height 3 pt depth -2.5pt
                  \kern -9pt \kern 1pt\intop}\nolimits_{\kern -5pt{#1}}}
          {\mathop{\vrule width 5pt height 3 pt depth -2.6pt
                  \kern -6pt \intop}\nolimits_{\kern -3pt{#1}}}
          {\mathop{\vrule width 5pt height 3 pt depth -2.6pt
                  \kern -6pt \intop}\nolimits_{\kern -3pt{#1}}}
          {\mathop{\vrule width 5pt height 3 pt depth -2.6pt
                  \kern -6pt \intop}\nolimits_{\kern -3pt{#1}}}}

\newcommand{\eps}{\varepsilon}
\newcommand{\loc}{\mathrm{loc}}

\newcommand{\BV}{\mathrm{BV}}
\newcommand{\liploc}{\mathrm{Lip}_{\mathrm{loc}}}
\newcommand{\mF}{\mathcal F}

\DeclareMathOperator{\dist}{dist}

\DeclareMathOperator{\Lip}{Lip}

\DeclareMathOperator{\supp}{supp}
\DeclareMathOperator*{\aplim}{\mathrm{ap\,lim}}

\begin{document}
\title[Integral representation for functionals]{Relaxation and integral representation for functionals of linear growth on metric measure spaces}
\author{Heikki Hakkarainen, Juha Kinnunen, Panu Lahti, Pekka Lehtel\"a}
\thanks{The research was supported by the Academy of Finland and the Finnish Academy of Science and Letters, Vilho, Yrj\"o and Kalle V\"ais\"al\"a Foundation. Part of the work was done during a visit at the Institut Mittag-Leffler (Djursholm, Sweden).}
\subjclass[2010]{49Q20, 30L99, 26B30}

\begin{abstract} 
This article studies an integral representation of functionals of linear growth on metric measure spaces with a doubling measure and a Poincar\'e inequality. 
Such a functional is defined through relaxation, and it defines a Radon measure on the space.
For the singular part of the functional, we get the expected integral representation with respect to the variation measure.
A new feature is that in the representation for the absolutely continuous part, a constant appears already in the weighted Euclidean case.
As an application we show that in a variational minimization problem related to the functional, boundary values can be presented as a penalty term.
\end{abstract}

\maketitle

\section{Introduction}
Let $f:\R_+\to\R_+$ be a convex, nondecreasing function that satisfies the linear growth condition
\[
mt\le f(t)\le M(1+t)
\]
with some constants $0<m\le M<\infty$.
Let $\Omega$ be an open set on a metric measure space $(X,d,\mu)$. Throughout the work we assume that the measure is doubling and that the space supports a Poincar\'e inequality. For $u\in L^1_{\loc}(\Omega)$, we define the functional of linear growth via relaxation by
\begin{align*}
&\mathcal F(u,\Omega)\\
&\quad=\inf\left\{\liminf_{i\to\infty}\int_{\Omega}f(g_{u_i})\,d\mu:\,u_i\in \liploc(\Omega),\,u_i\to u\text{ in }L^1_{\loc}(\Omega)\right\},
\end{align*}
where $g_{u_i}$ is the minimal 1-weak upper gradient of $u_i$.
For $f(t)=t$, this is the definition of functions of bounded variation, or $\BV$ functions, on metric measure spaces, see \cite{A2}, \cite{AMP} and \cite{M}.
For $f(t)=\sqrt{1+t^2}$, we get the generalized surface area functional, see \cite{HKL}.  
Our first result shows that if $\mF(u,\Omega)<\infty$, then $\mF(u,\cdot)$ is a Borel regular outer measure on $\Omega$.
This result is a generalization of \cite[Theorem 3.4]{M}.
For corresponding results in the Euclidean case with the Lebesgue measure, we refer to \cite{AmbFP00}, \cite{ButGH98}, \cite{GiaMS98I}, \cite{GiaMS98II}, and \cite{GiaMS79}.

Our second goal is study whether the relaxed functional $\mathcal F(u,\cdot)$ can be represented as an integral. 
To this end, let $u\in L_{\loc}^1(\Omega)$ with $\mF(u,\Omega)<\infty$. Then the growth condition implies that $u\in\BV(\Omega)$.
We denote the  decomposition of the variation measure $\Vert Du\Vert$  into the absolute continuous and singular parts
by $d\Vert Du\Vert=a\, d\mu+d\Vert Du \Vert ^s$, where $a \in L^1(\Omega)$. 
Similarly, we denote by  $\mF^a(u,\cdot)$ and $\mF^s(u,\cdot)$ the absolutely continuous and singular parts of $\mF(u,\cdot)$ with respect to $\mu$.
For the singular part, we obtain the integral representation
\[
\mF^s(u,\Omega)=f_{\infty}\Vert Du\Vert^s(\Omega),
\]
where  $f_{\infty}=\lim_{t\to\infty}f(t)/t$.  This is analogous to the Euclidean case.
However, for the absolutely continuous part we only get an integral representation up to a constant
\[
\int_{\Omega}f(a)\,d\mu \le \mF^a(u,\Omega)\le \int_{\Omega}f(Ca)\,d\mu,
\]
where $C$ depends on the doubling constant of the measure and the constants in the Poincar\'e inequality.
Furthermore, we give a counterexample which shows that the constant cannot be dismissed. 
We observe that working in the general metric context produces significant challenges that are already visible in the Euclidean setting with a weighted Lebesgue measure.
In overcoming these challenges, a key technical tool is an equi-integrability result for the discrete convolution of a measure. 
As a by-product of our analysis, we are able to show that a $\BV$  function is actually a Newton-Sobolev function in a set where the variation measure is absolutely continuous.

As an application of the integral representation, we consider a minimization problem related to functionals of linear growth.
First we define the concept of boundary values of $\BV$ functions, which is a delicate issue already in the Euclidean case.
Let $\Omega\Subset\Omega^*$ be bounded open subsets of $X$, and assume that $h\in\BV(\Omega^*)$. 
We define $\BV_{h}(\Omega)$ as the
space of functions $u\in\BV(\Omega^*)$ such that $u=h$ $\mu$-almost everywhere in $\Omega^*\setminus\Omega$.
A function $u\in \BV_{h}(\Omega)$  is a minimizer of the functional of linear growth
with boundary values $h$, if
\[
\mathcal F(u,\Omega^*)= \inf\mathcal F(v,\Omega^*),
\]
where the infimum is taken over all $v\in\BV_h(\Omega)$. It was shown in \cite{HKL} that this problem always has a solution.
By using the integral representation, we can express the boundary values as a penalty term. More precisely, under suitable conditions on the space and $\Omega$, we establish equivalence between the above minimization problem and minimizing the functional
\[
\mathcal F(u,\Omega)+f_{\infty}\int_{\partial \Omega}|T_\Omega u-T_{X\setminus\Omega}h|\theta_\Omega \,d\mathcal H
\]
over all $u\in\BV(\Omega)$. Here $T_\Omega u$ and $T_{X\setminus\Omega}u$ are boundary traces and $\theta_\Omega$
is a strictly positive density function. 
This is the main result of the paper, and it extends the Euclidean results in \cite[p. 582]{GiaMS98II} to metric measure spaces.
A careful analysis of $\BV$ extension domains and boundary traces is needed in the argument.

\section{Preliminaries}\label{sec:prelis}
In this paper, $(X,d,\mu)$ is a complete metric measure space 
with a Borel regular outer measure $\mu$.
The measure $\mu$ is assumed to be doubling, meaning that there exists a constant $c_d>0$ such that
\[
0<\mu(B(x,2r))\leq c_d\mu(B(x,r))<\infty
\]
for every ball $B(x,r)$ with center $x\in X$ and radius $r > 0$. For brevity, we will sometimes write $\lambda B$ for $B(x,\lambda r)$. On a metric space, a ball $B$ does not necessarily have a unique center point and radius, but we assume every ball to come with a prescribed center and radius.
The doubling condition implies that
\begin{equation}\label{eq:doubling dimension}
\frac{\mu(B(y,r))}{\mu(B(x,R))}\ge C\left(\frac{r}{R}\right)^Q
\end{equation}
for every $r\leq R$ and $y\in B(x,R)$, and some $Q>1$ and $C\ge1$ that only depend on $c_d$.
We recall that a complete metric space endowed with a doubling measure is proper,
that is, closed and bounded sets are compact. Since $X$ is proper, for any open set $\Omega\subset X$
we define  $\textrm{Lip}_{\loc}(\Omega)$ as the space of
functions that are Lipschitz continuous in every $\Omega'\Subset\Omega$ (and other local spaces of functions are defined similarly).
Here $\Omega'\Subset\Omega$ means that $\Omega'$ is open and that $\overline{\Omega'}$ is a
compact subset of $\Omega$.


For any set $A\subset X$, the restricted spherical Hausdorff content
of codimension $1$ is defined as
\[
\mathcal{H}_{R}(A)=\inf\left\{ \sum_{i=1}^{\infty}\frac{\mu(B(x_{i},r_{i}))}{r_{i}}:\, A\subset\bigcup_{i=1}^{\infty}B(x_{i},r_{i}),\, r_{i}\le R\right\},
\]
where $0<R<\infty$.
The Hausdorff measure of codimension $1$ of a set
$A\subset X$ is
\[
\mathcal{H}(A)=\lim_{R\rightarrow0}\mathcal{H}_{R}(A).
\]

The measure theoretic boundary $\partial^{*}E$ is defined as the set of points $x\in X$
in which both $E$ and its complement have positive density, i.e.
\[
\limsup_{r\rightarrow0}\frac{\mu(B(x,r)\cap E)}{\mu(B(x,r))}>0\quad\;\textrm{and}\quad\;\limsup_{r\rightarrow0}\frac{\mu(B(x,r)\setminus E)}{\mu(B(x,r))}>0.
\]

A curve $\gamma$ is a rectifiable continuous mapping from a compact interval
to $X$. The length of a curve $\gamma$
is denoted by $\ell_{\gamma}$. We will assume every curve to be parametrized
by arc-length, which can always be done (see e.g. \cite[Theorem 3.2]{Hj}).

A nonnegative Borel function $g$ on $X$ is an upper gradient 
of an extended real-valued function $u$
on $X$ if for all curves $\gamma$ in $X$, we have
\begin{equation} \label{ug-cond}
|u(x)-u(y)|\le \int_\gamma g\,ds
\end{equation}
whenever both $u(x)$ and $u(y)$ are finite, and 
$\int_\gamma g\, ds=\infty $ otherwise.
Here $x$ and $y$ are the end points of $\gamma$.
If $g$ is a nonnegative $\mu$-measurable function on $X$
and (\ref{ug-cond}) holds for $1$-almost every curve,
then $g$ is a $1$-weak upper gradient of~$u$. 
A property holds for $1$-almost every curve
if it fails only for a curve family with zero $1$-modulus. 
A family $\Gamma$ of curves is of zero $1$-modulus if there is a 
nonnegative Borel function $\rho\in L^1(X)$ such that 
for all curves $\gamma\in\Gamma$, the curve integral $\int_\gamma \rho\,ds$ is infinite.

We consider the following norm
\[
\Vert u\Vert_{N^{1,1}(X)}=\Vert u\Vert_{L^1(X)}+\inf_g\Vert g\Vert_{L^1(X)},
\]
where the infimum is taken over all upper gradients $g$ of $u$. 
The Newtonian space is defined as
\[
N^{1,1}(X)=\{u:\,\|u\|_{N^{1,1}(X)}<\infty\}/{\sim},
\]
where the equivalence relation $\sim$ is given by $u\sim v$ if and only if 
$\Vert u-v\Vert_{N^{1,1}(X)}=0$. In the definition of upper gradients and Newtonian spaces, the whole space $X$ can be replaced by any $\mu$-measurable (typically open) set $\Omega\subset X$. It is known that for any $u\in N_{\loc}^{1,1}(\Omega)$, there exists a minimal $1$-weak
upper gradient, which we always denote $g_{u}$, satisfying $g_{u}\le g$ 
$\mu$-almost everywhere in $\Omega$, for any $1$-weak upper gradient $g\in L_{\loc}^{1}(\Omega)$
 of $u$ \cite[Theorem 2.25]{BB}.
For more on Newtonian spaces, we refer to \cite{S} and \cite{BB}.

Next we recall the definition and basic properties of functions
of bounded variation on metric spaces, see \cite{A2}, \cite{AMP} and \cite{M}. 
For $u\in L^1_{\text{loc}}(X)$, we define the total variation of $u$ as
\begin{align*}
&\|Du\|(X)\\
&\quad =\inf\left\{\liminf_{i\to\infty}\int_Xg_{u_i}\,d\mu:\, u_i\in \Lip_{\loc}(X),\, u_i\to u\textrm{ in } L^1_{\text{loc}}(X)\right\},
\end{align*}
where $g_{u_i}$ is the minimal $1$-weak upper gradient of $u_i$.
We say that a function $u\in L^1(X)$ is of bounded variation, 
and write $u\in\BV(X)$, if $\|Du\|(X)<\infty$. 
Moreover, a $\mu$-measurable set $E\subset X$ is said to be of finite perimeter if $\|D\chi_E\|(X)<\infty$.
By replacing $X$ with an open set $\Omega\subset X$ in the definition of the total variation, we can define $\|Du\|(\Omega)$.
For an arbitrary set $A\subset X$, we define
\[
\|Du\|(A)=\inf\{\|Du\|(\Omega):\, A\subset\Omega,\,\Omega\subset X
\text{ is open}\}.
\]
If $u\in\BV(\Omega)$, $\|Du\|(\cdot)$ is a finite Radon measure on $\Omega$ by \cite[Theorem 3.4]{M}.
The perimeter of $E$ in $\Omega$ is denoted by
\[
P(E,\Omega)=\|D\chi_E\|(\Omega).
\]
We have the following coarea formula given by Miranda in \cite[Proposition 4.2]{M}: if $\Omega\subset X$ is an open set and $u\in L_{\loc}^{1}(\Omega)$, then
\begin{equation}\label{eq:coarea}
\|Du\|(\Omega)=\int_{-\infty}^{\infty}P(\{u>t\},\Omega)\,dt.
\end{equation}
For an open set $\Omega\subset X$ and a set of locally finite perimeter $E\subset X$, we know that
\begin{equation}\label{eq:def of theta}
\Vert D\chi_{E}\Vert(\Omega)=\int_{\partial^{*}E\cap \Omega}\theta_E\,d\mathcal H,
\end{equation}
where $\theta_E:X\to [\alpha,c_d]$, with $\alpha=\alpha(c_d,c_P)>0$, see \cite[Theorem 5.3]{A2} and \cite[Theorem 4.6]{AMP}. The constant $c_P$ is related to the Poincar\'e inequality, see below.

The jump set of a function $u\in\BV_{\loc}(X)$ is defined as
\[
S_{u}=\{x\in X:\,u^{\wedge}(x)<u^{\vee}(x)\},
\]
where $u^{\wedge}$ and $u^{\vee}$ are the lower and upper approximate limits of $u$ defined as
\[
u^{\wedge}(x)
=\sup\left\{t\in\overline\R:\,\lim_{r\to0}\frac{\mu(\{u<t\}\cap B(x,r))}{\mu(B(x,r))}=0\right\}
\]
and
\[
u^{\vee}(x)
=\inf\left\{t\in\overline\R:\,\lim_{r\to0}\frac{\mu(\{u>t\}\cap B(x,r))}{\mu(B(x,r))}=0\right\}.
\]
Outside the jump set, i.e. in $X\setminus S_u$, $\mathcal H$-almost every point is a Lebesgue point of $u$ \cite[Theorem 3.5]{KKST}, and we denote the Lebesgue limit at $x$ by $\widetilde{u}(x)$.

%

We say that $X$ supports a $(1,1)$-Poincar\'e inequality
if there exist constants $c_P>0$ and $\lambda \ge1$ such that for all
balls $B(x,r)$, all locally integrable functions $u$,
and all $1$-weak upper gradients $g$ of $u$, we have 
\[
\vint{B(x,r)}|u-u_{B(x,r)}|\, d\mu 
\le c_P r\,\vint{B(x,\lambda r)}g\,d\mu,
\]
where 
\[
u_{B(x,r)}=\vint{B(x,r)}u\,d\mu =\frac 1{\mu(B(x,r))}\int_{B(x,r)}u\,d\mu.
\]
If the space supports a $(1,1)$-Poincar\'e inequality, by an approximation argument we get for
every $u\in L^1_{\loc}(X)$
\[
\vint{B(x,r)} |u-u_{B(x,r)}|\, d\mu 
\le c_P r\frac{\|Du\|(B(x,\lambda r))}{\mu(B(x,\lambda r))},
\]
where the constant $c_P$ and the dilation factor $\lambda$ are the same as in the  $(1, 1)$-Poincar\'e inequality. When $u=\chi_{E}$ for $E\subset X$, we get the relative isoperimetric inequality
\begin{equation}\label{eq:isop ineq}
\min\{\mu(B(x,r)\cap E), \mu(B(x,r)\setminus E)\} 
\le 2c_P r\|D\chi_{E}\|(B(x,\lambda r)).
\end{equation}
\emph{Throughout the work we assume, without further notice,  that the measure $\mu$ is doubling and that
the space supports a  $(1, 1)$-Poincar\'e inequality.}

\section{Functional and its measure property}\label{sec:functional}

In this section we define the functional that is considered in this paper, and show that it defines a Radon measure. Let $f$ be a convex nondecreasing function that is defined on $[0,\infty)$ and satisfies the linear growth condition
\begin{equation}\label{eq:linear growth}
mt\le f(t)\le M(1+t)
\end{equation}
for all $t\ge 0$, with some constants $0<m\le M<\infty$. This implies that $f$ is Lipschitz continuous with constant $L>0$. Furthermore, we define
\[
f_{\infty}=\sup_{t>0} \frac{f(t)-f(0)}{t}=\lim_{t \to \infty} \frac{f(t)-f(0)}{t}=\lim_{t\to\infty}\frac{f(t)}{t},
\]
where the second equality follows from the convexity of $f$.
From the definition of $f_{\infty}$, we get the simple estimate
\begin{equation}\label{eq:estimate for f}
f(t)\le f(0)+tf_{\infty}
\end{equation}
for all $t\ge 0$. This will be useful for us later.

Now we give the definition of the functional. For an open set $\Omega$ and $u\in N^{1,1}(\Omega)$, we could define it as 
\[
u\longmapsto\int_{\Omega}f(g_u)\,d\mu, 
\]
where $g_{u}$ is the minimal 1-weak upper gradient of $u$. For $u\in \BV(\Omega)$, we need to use a relaxation procedure as given in the following definition.

\begin{definition}
Let $\Omega\subset X$ be an open set. For $u\in L^1_{\loc}(\Omega)$, we define
\begin{align*}
&\mathcal F(u,\Omega)\\
&\quad =\inf\left\{\liminf_{i\to\infty}\int_{\Omega}f(g_{u_i})\,d\mu:\,u_i\in \liploc(\Omega),\,u_i\to u\text{ in }L^1_{\loc}(\Omega)\right\},
\end{align*}
where $g_{u_i}$ is the minimal 1-weak upper gradient of $u_i$.
\end{definition}

Note that we could equally well require that $g_{u_i}$ is \emph{any} 1-weak upper gradient of $u_i$.
We define $\mathcal F(u,A)$ for an arbitrary set $A \subset X$ by
\begin{equation}\label{eq:def of F for general sets}
\mathcal F(u,A)=\inf\{\mF(u,\Omega): \,\Omega \textrm{ is open,}\,A\subset \Omega\}.
\end{equation}
In this section we show that if $u\in L_{\loc}^1(\Omega)$ with $\mathcal F(u,\Omega)<\infty$, then $\mF(u,\cdot)$ is a Borel regular outer measure on $\Omega$, extending \cite[Theorem 3.4]{M}. The functional clearly satisfies
\begin{equation}\label{eq:basic estimate for functional}
m\Vert Du \Vert(A) \le \mF(u,A) \le M(\mu(A)+\Vert Du\Vert(A))
\end{equation}
for any $A\subset X$. This estimate follows directly from the definition of the functional, the definition of the variation measure, and \eqref{eq:linear growth}.
It is also easy to see that
\[
\mF(u,B)\le \mF(u,A)
\]
for any sets $B\subset A\subset X$.

In order to show the measure property, we first prove a few lemmas. The first is the following technical gluing lemma that is similar to \cite[Lemma 5.44]{AmbFP00}.
\begin{lemma}\label{joining lemma}
Let $U'$, $U$, $V'$, $V$ be open sets in $X$ such that $U'\Subset U$ and $V'\subset V$. Then there exists an open set $H\subset (U\setminus U')\cap V'$, with $H\Subset U$, such that for any $\eps>0$ and any pair of functions $u\in \liploc(U)$ and $v\in \liploc(V)$, there is a function $\phi\in\Lip_{c}(U)$ with $0\le\phi\le 1$ and $\phi=1$ in a neighborhood of $U^{'}$, such that the function $w=\phi u+(1-\phi)v\in\liploc(U'\cup V')$ satisfies
\[
\int_{U'\cup V'} f(g_w)\,d\mu \le \int_U f(g_u)\,d\mu + \int_V f(g_v)\,d\mu + C\int_H |u-v|\,d\mu + \eps.
\]
Here $C=C(U,U',M)$.
\end{lemma}

\begin{proof}
Let $\eta = \dist(U',X\setminus U) >0$. 
Define 
\[
H=\left\{x\in U\cap V':\, \frac{\eta}{3} < \dist(x,U')<\frac{2\eta}{3}\right\}.
\]
Now fix $u\in \liploc(U),\,v\in \liploc(V)$ and $\eps>0$. Choose $k\in \N$ such that 
\begin{equation}\label{eq:gluing estimate Lipschitz}
M\int_H(1+g_u+g_v)\,d\mu < \eps k
\end{equation}
if the above integral is finite --- otherwise the desired estimate is trivially true.
For $i=1,\ldots,k$, define the sets
\[
H_i=\left\{x\in U\cap V':\, \frac{(k+i-1)\eta}{3k} < \dist(x,U') < \frac{(k+i)\eta}{3k}\right\},
\]
so that $H\supset\cup_{i=1}^k H_i$, and define the Lipschitz functions
\[
\phi_i(x)=\begin{cases}
0,                                          & \dist(x, U') > \frac{k+i}{3k} \eta, \\
\frac1\eta((k+i)\eta - 3k \dist(x,U')),\!\! & \frac{k+i-1}{3k}\eta \le \dist(x, U') \le \frac{k+i}{3k} \eta, \\
1,                                          & \dist(x, U') < \frac{k+i-1}{3k} \eta.
\end{cases}
\]
Now $g_{\phi_i}=0$ $\mu$-almost everywhere in $U\cap V'\setminus H_i$ \cite[Corollary 2.21]{BB}. Let $w_i=\phi_iu+(1-\phi_i)v$ on $U'\cup V'$. We have the estimate
\[
g_{w_i} \le \phi_i g_u + (1-\phi_i)g_v + g_{\phi_i} |u-v|,
\]
see \cite[Lemma 2.18]{BB}. By also using the estimate $f(t) \le M(1+t)$, we get
\begin{align*}
\int_{U'\cup V'} f(g_{w_i})\,d\mu &\le \int_U f(g_u)\,d\mu + \int_V f(g_v)\,d\mu + \int_{H_i} f(g_{w_i})\,d\mu \\
                                &\le \int_U f(g_u)\,d\mu + \int_V f(g_v)\,d\mu\\
                                &\quad\ +M\int_{H_i} (1+g_u+g_v)\,d\mu + \frac{3Mk}{\eta}\int_{H_i}|u-v|\,d\mu.
\end{align*}
Now, since $H\supset\cup_{i=1}^k H_i$, we have
\begin{align*}
\frac{1}{k}\sum_{i=1}^k& \int_{U'\cup V'} f(g_{w_i})\,d\mu\\
                      &\le \int_U f(g_u)\,d\mu + \int_V f(g_v)\,d\mu + \frac{M}{k} \int_H (1+g_u+g_v)\,d\mu \\
                      &\qquad+ \frac{3M}{\eta} \int_H |u-v|\,d\mu\\
                      &\le \int_U f(g_u)\,d\mu + \int_V f(g_v)\,d\mu+C \int_H|u-v|\,d\mu + \eps.
\end{align*}
In the last inequality we used \eqref{eq:gluing estimate Lipschitz}. Thus we can find an index $i$ such that the function $w=w_i$ satisfies the desired estimate.
\end{proof}

In the following lemmas, we assume that $u\in L_{\loc}^1(A\cup B)$.

\begin{lemma}\label{inner regularity lemma}
Let $A\subset X$ be open with $\mF(u,A)<\infty$. Then
\[
\mF(u,A)= \sup_{B\Subset A} \mF(u,B).
\]
\end{lemma}
\begin{proof}
Take open sets $B_1\Subset B_2 \Subset B_3 \Subset A$ and sequences $u_i\in \liploc(B_3)$, $v_i\in \liploc(A\setminus \overline{B_1})$ such that $u_i\to u$ in $L^1_\loc(B_3)$, $v_i\to u$ in $L^1_\loc(A\setminus \overline{B_1})$,
\[
\mF(u,B_3)= \lim_{i\to \infty} \int_{B_3} f(g_{u_i}) \,d\mu,
\]
and
\[
\mF(u,A\setminus \overline{B_1}) = \lim_{i\to \infty} \int_{A\setminus \overline{B_1}} f(g_{v_i}) \,d\mu.
\]
By using Lemma \ref{joining lemma} with $U=B_3$, $U'=B_2$, $V=V'=A\setminus \overline{B_1}$ and $\eps=1/i$, we find a set $H\subset B_3\setminus B_2$, $H\Subset B_3$, and a sequence $w_i \in \liploc (A)$ such that $w_i\to u$ in $L^1_{\loc}(A)$, and
\[
\int_A f(g_{w_i})\,d\mu \le \int_{B_3}f(g_{u_i}) \,d\mu + \int_{A\setminus \overline{B_1}} f(g_{v_i}) \,d\mu + C \int_H|u_i-v_i| \,d\mu + \frac{1}{i}
\]
for every $i\in\N$.
In the above inequality, the last integral converges to zero as $i\to \infty$, since $H\Subset B_3$ and $H\Subset A\setminus \overline{B_1}$. Thus
\[
\mF(u,A)\le \liminf_{i\to\infty}\int_A f(g_{w_i})\,d\mu \le \mF(u,B_3)+\mF(u,A\setminus \overline{B_1}).
\]
Exhausting $A$ with sets $B_1$ concludes the proof, since then  $\mF(u,A\setminus \overline{B_1})\to 0$ by \eqref{eq:basic estimate for functional}.
\end{proof}

\begin{lemma}\label{subadditivity lemma}
Let $A,B\subset X$ be open. Then
\[
\mF(u,A\cup B) \le \mF(u,A)+\mF(u,B).
\]
\end{lemma}
\begin{proof}
First we note that
every $C\Subset A\cup B$ can be presented as $C=A'\cup B'$, where $A'\Subset A$ and $B'\Subset B$.
Therefore, according to Lemma \ref{inner regularity lemma}, it suffices to show that
\[
\mF(u,A'\cup B') \le \mF(u,A)+\mF(u,B) 
\]
for every $A'\Subset A$ and $B'\Subset B$. If $\mF(u,A)=\infty$ or $\mF(u,B)=\infty$, the claim holds.
Assume therefore that $\mF(u,A)<\infty$ and $\mF(u,B)<\infty$.
Take sequences $u_i \in \liploc(A)$ and $v_i\in \liploc(B)$ such that $u_i\to u$ in $L^1_\loc(A)$, $v_i\to u$ in $L^1_\loc(B)$,
\[
\mF(u,A)=\lim_{i\to \infty} \int_A f(g_{u_i})\,d\mu,
\]
and
\[
\mF(u,B)=\lim_{i\to \infty} \int_B f(g_{v_i})\,d\mu.
\]
By using Lemma \ref{joining lemma} with $U'=A'$, $U=A$, $V'=B'$, $V=B$ and $\eps=1/i$, we find a set $H\Subset A$, $H\subset B'\Subset B$, and a sequence $w_i\in \liploc(A'\cup B')$ such that $w_i\to u$ in $L^1_\loc(A'\cup B')$, and
\[
\int_{A'\cup B'} f(g_{w_i})\,d\mu \le \int_A f(g_{u_i})\,d\mu + \int_B f(g_{v_i})\,d\mu+C\int_H|u_i-v_i|\,d\mu + \frac{1}{i}
\]
for every $i\in\N$. By the properties of $H$, the last integral in the above inequality converges to zero as $i\to \infty$, and then
\[
\mF(u,A'\cup B') \le \mF(u,A)+\mF(u,B). 
\]
\end{proof}

\begin{lemma}\label{lem:additivity lemma}
Let $A,B\subset X$ be open and let $A\cap B=\emptyset$. Then 
\[
\mF(u,A\cup B) \ge \mF(u,A)+\mF(u,B).
\]
\end{lemma}
\begin{proof}
If $\mF(u,A\cup B)=\infty$, the claim holds. Hence we may assume that $\mF(u,A\cup B)<\infty$. Take a sequence $u_i\in \liploc(A\cup B)$ such that $u_i\to u$ in $L^1_\loc (A\cup B)$ and 
\[
\mF(u,A\cup B) = \lim_{i\to \infty} \int_{A\cup B} f(g_{u_i})\,d\mu.
\]
Then, since $A$ and $B$ are disjoint,
\begin{align*}
\mF(u,A\cup B) &= \lim_{i\to \infty} \int_{A\cup B} f(g_{u_i})\,d\mu \\
&\ge \liminf_{i\to \infty} \int_{A} f(g_{u_i})\, d\mu + \liminf_{i\to \infty} \int_{B} f(g_{u_i})\,d\mu\\
&\ge \mF(u,A)+\mF(u,B).
\end{align*}
\end{proof}

Now we are ready to prove the measure property of the functional.

\begin{theorem}\label{thm:measure_prop}
Let $\Omega\subset X$ be an open set, and let $u\in L^1_{\loc}(\Omega)$ with $\mathcal F(u,\Omega)<\infty$. Then $\mathcal F(u,\cdot)$ is a Borel regular outer measure on $\Omega$.
\end{theorem}
\begin{proof}
First we show that $\mF(u,\cdot)$ is an outer measure on $\Omega$. Obviously $\mF(u,\emptyset)=0$. As mentioned earlier, clearly $\mF(u,A)\le \mF(u,B)$ for any $A\subset B\subset \Omega$.
Take open sets $A_i\subset \Omega$, $i=1,2,\ldots$.  Let $\eps >0$. By Lemma \ref{inner regularity lemma} there exists a set $B\Subset \cup_{i=1}^\infty A_i$ such that 
\[
\mF\left(u,\bigcup_{i=1}^\infty A_i\right) < \mF(u,B) + \eps.
\]
Since $\overline{B}\subset \cup_{i=1}^\infty A_i$ is compact, there exists $n\in \N$ such that 
$B\subset \overline{B} \subset \cup_{i=1}^n A_i$.
Then by Lemma \ref{subadditivity lemma},
\[
\mF(u,B) \le \mF \left(u, \bigcup_{i=1}^n A_i \right) \le \sum_{i=1}^n \mF(u,A_i),
\]
and thus letting $n\to \infty$ and $\eps \to 0$ gives us
\begin{equation}\label{eq: countable subadditivity}
\mF\bigg(u,\bigcup_{i=1}^\infty A_i\bigg) \le \sum_{i=1}^\infty \mF(u,A_i).
\end{equation}
For general sets $A_i$, we can prove \eqref{eq: countable subadditivity} by approximation with open sets. 

The next step is to prove that $\mF(u,\cdot)$ is a Borel outer measure. Let $A,B\subset \Omega$ satisfy $\dist(A,B)>0$. Fix $\eps>0$ and choose an open set $U\supset A\cup B$ such that 
\[
\mF(u,A\cup B) > \mF(u,U)-\eps.
\]
Define the sets
\begin{align*}
&V_A= \left\{x\in \Omega :\, \dist(x,A) < \frac{\dist(A,B)}{3}\right \} \cap U,\\
&V_B= \left\{x\in \Omega :\, \dist(x,B) < \frac{\dist(A,B)}{3}\right \} \cap U.
\end{align*}
Then $V_A, V_B$ are open and $A\subset V_A$, $B\subset V_B$. Moreover $V_A\cap V_B=\emptyset$. Thus by Lemma \ref{lem:additivity lemma},
\begin{align*}
\mF(u,A\cup B) &\ge \mF(u,V_A\cup V_B)-\eps \\
               &\ge\mF(u,V_A)+\mF(u,V_B) -\eps\\
               &\ge \mF(u,A)+\mF(u,B)-\eps.
\end{align*}
Now letting $\eps\to 0$ shows that $\mF(u,\cdot)$ is a Borel outer measure by Carath{\' e}odory's criterion.

The measure $\mF(u,\cdot)$ is Borel regular by construction, since for every $A\subset \Omega$ we may choose open sets $V_i$ such that $A\subset V_i \subset \Omega$ and 
\[
\mF(u,V_i)<\mF(u,A)+\frac{1}{i},
\]
and by defining $V=\cap_{i=1}^\infty V_i$, we get $\mF(u,V)=\mF(u,A)$,
where $V\supset A$ is a Borel set.
\end{proof}

As a simple application of the measure property of the functional, we show the following approximation result.

\begin{proposition}\label{prop:weak convergence}
Let $\Omega\subset X$ be an open set, and let $u\in L^1_{\loc}(\Omega)$ with $\mathcal F(u,\Omega)<\infty$. Then for any sequence of functions $u_i\in \liploc(\Omega)$ for which $u_i\to u$ in $L_{\loc}^1(\Omega)$ and
\[
\int_{\Omega}f(g_{u_i})\,d\mu\to \mathcal F(u,\Omega),
\]
we also have $f(g_{u_i})\,d\mu \overset{*}{\rightharpoonup} d\mathcal F(u,\cdot)$ in $\Omega$.
\end{proposition}
\begin{proof}
For any open set $U\subset \Omega$, we have by the definition of the functional that
\begin{equation}\label{eq:weak convergence open set}
\mathcal F(u,U)\le \liminf_{i\to\infty}\int_U f(g_{u_i})\,d\mu.
\end{equation}
On the other hand, for any relatively closed set $F\subset\Omega$ we have
\begin{align*}
\mathcal F(u,\Omega) &= \limsup_{i\to\infty}\int_{\Omega}f(g_{u_i})\,d\mu\\
  &\ge \limsup_{i\to\infty}\int_{F}f(g_{u_i})\,d\mu+\liminf_{i\to\infty}\int_{\Omega\setminus F}f(g_{u_i})\,d\mu\\
  &\ge \limsup_{i\to\infty}\int_{F}f(g_{u_i})\,d\mu+\mathcal F(u,\Omega\setminus F).
\end{align*}
The last inequality follows from the definition of the functional, since $\Omega\setminus F$ is open. By the measure property of the functional, we can subtract $\mathcal F(u,\Omega\setminus F)$ from both sides to get
\[
\limsup_{i\to\infty}\int_{F}f(g_{u_i})\,d\mu\le \mathcal F(u,F).
\]
According to a standard characterization of the weak* convergence of Radon measures, the above inequality and \eqref{eq:weak convergence open set} together give the result \cite[p. 54]{EvGa}.
\end{proof}

\section{Integral representation}

In this section we study an integral representation for the functional $\mathcal F(u,\cdot)$.
First we show the estimate from below. Note that due to \eqref{eq:basic estimate for functional}, $\mathcal F(u,\Omega)<\infty$ always implies $\Vert Du\Vert(\Omega)<\infty$.
\begin{theorem}\label{thm:estimate from below}
Let $\Omega$ be an open set, and let $u\in L^1_{\loc}(\Omega)$ with $\mathcal F(u,\Omega)<\infty$. Let $d\Vert Du\Vert=a\,d\mu+d\,\Vert Du\Vert^s$ be the decomposition of the variation measure into the absolutely continuous and singular parts, where $a\in L^1(\Omega)$ is a Borel function and $\Vert Du\Vert^s$ is the singular part. Then we have
\[
\mathcal F(u,\Omega)\ge\int_{\Omega}f(a)\,d\mu+f_{\infty}\Vert Du\Vert^s(\Omega).
\]
\end{theorem}
\begin{proof}
Pick a sequence $u_i\in\liploc(\Omega)$ such that $u_i\to u$ in $L^1_{\loc}(\Omega)$ and
\begin{equation}\label{eq:choice of sequence}
\int_{\Omega}f(g_{u_i})\,d\mu\to \mF(u,\Omega)\quad \textrm{as}\ \ i\to\infty.
\end{equation}
Using the linear growth condition for $f$, presented in \eqref{eq:linear growth}, we estimate
\[
\limsup_{i\to\infty}\int_{\Omega} g_{u_i}\,d\mu\le \frac{1}{m}\limsup_{i\to \infty}\int_{\Omega} f(g_{u_i})\,d\mu< \infty. 
\]
Picking a suitable subsequence, which we still denote $g_{u_i}$, we have $g_{u_i}\,d\mu\overset{*}{\rightharpoonup}d\nu$ in $\Omega$, where $\nu$ is a Radon measure with finite mass in $\Omega$. Furthermore, by the definition of the variation measure, we necessarily have $\nu\ge \Vert Du\Vert$, which can be seen as follows. For any open set $U\subset \Omega$ and for any $\eps>0$, we can pick an open set $U'\Subset U$ such that $\Vert Du\Vert(U)<\Vert Du\Vert(U')+\eps$; see e.g. Lemma \ref{inner regularity lemma}. We obtain
\begin{align*}
\Vert Du\Vert(U)&<\Vert Du\Vert(U')+\eps\le \liminf_{i\to \infty} \int_{U'}g_{u_i}\,d\mu+\eps\\
&\le \limsup_{i\to \infty}\int_{\overline{U'}}g_{u_i}\,d\mu+\eps
\le \nu(\overline{U'})+\eps
\le \nu(U)+\eps.
\end{align*}
On the first line we used the definition of the variation measure, and on the second line we used a property of the weak* convergence of Radon measures, see e.g. \cite[Example 1.63]{AmbFP00}. By approximation we get $\nu(A)\ge \Vert Du\Vert(A)$ for any $A\subset \Omega$.

The following lower semicontinuity argument is from \cite[p. 64--66]{AmbFP00}. First we note that as a nonnegative nondecreasing convex function, $f$ can be presented as
\[
f(t)=\sup_{j\in \N}(d_jt+e_j),\quad t\ge 0,
\]
for some sequences $d_j,e_j\in\R$, with $d_j\ge 0$, $j=1,2,\ldots$, and furthermore $\sup_jd_j=f_{\infty}$ \cite[Proposition 2.31, Lemma 2.33]{AmbFP00}.
Given any pairwise disjoint open subsets of $\Omega$, denoted $A_1,\ldots,A_k$, $k\in\N$, and functions $\varphi_j\in C_c(A_j)$ with $0\le \varphi_j\le 1$, we have
\[
\int_{A_j}(d_jg_{u_i}+e_j)\varphi_j\,d\mu \le \int_{A_j}f(g_{u_i})\,d\mu.
\]
for every $j=1,\ldots,k$ and $i\in\N$.
Summing over $j$ and letting $i\to \infty$, we get by the weak* convergence $g_{u_i}\,d\mu\overset{*}{\rightharpoonup}d\nu$
\[
\sum_{j=1}^k\left(\int_{A_j}d_j\varphi_j \,d\nu +\int_{A_j}e_j\varphi_j\,d\mu\right) \le \liminf_{i\to\infty}\int_{\Omega} f(g_{u_i})\,d\mu.
\]
Since we had $\nu\ge \Vert Du\Vert$, this immediately implies
\[
\sum_{j=1}^k\left(\int_{A_j}d_j\varphi_j \,d\Vert Du\Vert +\int_{A_j}e_j\varphi_j\,d\mu\right) \le \liminf_{i\to\infty}\int_{\Omega} f(g_{u_i})\,d\mu.
\]
We recall that $d\Vert Du\Vert=a\,d\mu+d\Vert Du\Vert^s$. It is known that the singular part $\Vert Du \Vert^s$ is concentrated on a Borel set $D \subset \Omega$ that satisfies $\mu (D)=0$ and $\Vert Du\Vert^s (\Omega \setminus D)=0$, see e.g. \cite[p. 42]{EvGa}. Define the Radon measure $\sigma=\mu+\Vert Du\Vert^s$, and the Borel functions
\[
\phi_j=\begin{cases}
d_ja+e_j,                \quad & \text{on }\Omega\setminus D,\\
d_j,                        \quad & \text{on } D
\end{cases}
\]
for $j=1,\ldots,k$, and
\[
\phi=\begin{cases}
f(a),                       \quad & \text{on }\Omega\setminus D,\\
f_{\infty},                  \quad & \text{on } D.
\end{cases}
\]
As mentioned above, we now have $\sup_j\phi_j=\phi$, and we can write the previous inequality as
\[
\sum_{j=1}^k\int_{A_j}\phi_j\varphi_j\,d\sigma \le \liminf_{i\to\infty}\int_{\Omega} f(g_{u_i})\,d\mu.
\]
Since the functions $\varphi_j\in C_c(A_j)$, $0\le \varphi_j\le 1$, were arbitrary, we get
\[
\sum_{j=1}^k\int_{A_j}\phi_j\,d\sigma\le \liminf_{i\to\infty}\int_{\Omega} f(g_{u_i})\,d\mu.
\]
Since this holds for \emph{any} pairwise disjoint open subsets $A_1,\ldots,A_k\subset \Omega$, by \cite[Lemma 2.35]{AmbFP00} we get
\[
\int_{\Omega} \phi \,d\sigma\le \liminf_{i\to\infty}\int_{\Omega} f(g_{u_i})\,d\mu.
\]
However, by the definitions of $\phi$ and $\sigma$, this is the same as
\[
\int_{\Omega}f(a)\,d\mu+f_{\infty}\Vert Du\Vert^s(\Omega)\le \liminf_{i\to\infty}\int_{\Omega} f(g_{u_i})\,d\mu.
\]
Combining this with \eqref{eq:choice of sequence}, we get the desired estimate from below.
\end{proof}

It is worth noting that in the above argument, we only needed the weak* convergence of the sequence $g_{u_i}\,d\mu$ to a Radon measure that majorizes $\Vert Du\Vert$. Then we could use the fact that the functional for measures
\[
\nu \longmapsto \int_{\Omega}f(\check{a})\,d\mu+f_{\infty}\nu^s(\Omega),\quad\ d\nu=\check{a}\,d\mu+d\nu^s,
\]
is lower semicontinuous with respect to weak* convergence of Radon measures. This \emph{lower} semicontinuity is guaranteed by the fact that $f$ is convex, but in order to have \emph{upper} semicontinuity, we should have that $f$ is also concave (and thus linear).
Thus there is an important asymmetry in the setting, and for the estimate from above, we will need to use rather different methods where we prove weak or strong $L^1$-convergence for the sequence of upper gradients, instead of just weak* convergence of measures.
To achieve this type of stronger convergence, we need to specifically ensure that the sequence of upper gradients is \emph{equi-integrable}. The price that is paid is that a constant $C$ appears in the final estimate related to the absolutely continuous parts. An example that we provide later shows that this constant cannot be discarded.

We recall that for a $\mu$-measurable set $F\subset X$, the equi-integrability of a sequence of functions $g_i\in L^1(F)$, $i\in\N$, is defined by two conditions. First, for any $\eps>0$, there must be a $\mu$-measurable set $A\subset F$ with $\mu(A)<\infty$ such that 
\[
\int_{F\setminus A}g_i\,d\mu<\eps \quad\textrm{for all }i\in\N.
\]
Second, for any $\eps>0$ there must be $\delta>0$ such that if $A\subset F$ is $\mu$-measurable with $\mu(A)<\delta$, then
\[
\int_{A}g_i\,d\mu<\eps \quad\textrm{for all }i\in\N.
\]

We will need the following equi-integrability result that partially generalizes \cite[Lemma 6]{FHK}. For the construction of Whitney coverings that are needed in the result, see e.g. \cite[Theorem 3.1]{BBS07}.

\begin{lemma}\label{lem:generalized equiintegrability}
Let $\Omega\subset X$ be open, let $F\subset\Omega$ be $\mu$-measurable, and let $\nu$ be a Radon measure with finite mass in $\Omega$. Write the decomposition of $\nu$ into the absolutely continuous and singular parts with respect to $\mu$ as $d\nu=a\,d\mu+d\nu^s$, and assume that $\nu^s(F)=0$. Take a sequence of open sets $F_i$ such that $F\subset F_i\subset \Omega$ and $\nu^s(F_i)<1/i$, $i\in\N$. For a given $\tau\ge 1$ and every $i\in\N$, take a Whitney covering $\{B_j^i=B(x_j^i,r_j^i)\}_{j=1}^{\infty}$ of $F_i$ such that $r_j^i\le 1/i$ for every $j\in\N$, $\tau B_j^i\subset F_i$ for every $j\in\N$, every ball $\tau B_k^i$ meets at most $c_o=c_o(c_d,\tau)$ balls $\tau B_j^i$, and if $\tau B_j^i$ meets $\tau B_k^i$, then $r_j^i\le 2r_k^i$. Define the functions
\[
g_i=\sum_{j=1}^{\infty}\chi_{B_j^i}\frac{\nu(\tau B_j^i)}{\mu(B_j^i)},\quad\ i\in\N.
\]
Then the sequence $g_i$ is equi-integrable in $F$. Moreover, a subsequence of $g_i$ converges weakly in $L^1(F)$ to a function $\check{a}$ that satisfies $\check{a}\le c_oa$ $\mu$-almost everywhere in $F$.
\end{lemma}
\begin{remark}
If the measure $\nu$ is absolutely continuous in the whole of $\Omega$, then we can choose $F=F_i=\Omega$ for all $i\in\N$.
\end{remark}
\begin{proof}
To check the first condition of equi-integrability, let $\eps>0$ and take a ball $B=B(x_0,R)$ with $x_0\in X$ and $R>0$ so large that $\nu(\Omega\setminus B(x_0,R))<\eps/c_o$. Then, by the bounded overlap property of the Whitney balls, we have
\[
\int_{F\setminus B(x_0,R+2\tau)}g_i\,d\mu\le c_o\nu(F_i\setminus B(x_0,R))<\eps
\]
for all $i\in\N$.

To check the second condition, assume by contradiction that there is a sequence of $\mu$-measurable sets $A_i\subset F$ with $\mu(A_i)\to 0$, and $\int_{A_i}g_i\,d\mu>\eta>0$ for all $i\in\N$. Fix $\eps>0$. We know that there is $\delta>0$ such that if $A\subset \Omega$ and $\mu(A)<\delta$, then $\int_Aa\,d\mu<\eps$. Note that by the bounded overlap property of the Whitney balls, we have for every $i\in\N$
\begin{equation}\label{eq:equiintegrability first estimate}
\begin{split}
\int_{A_i}g_i\,d\mu&= \sum_{j=1}^{\infty}\frac{\mu(A_i\cap B_j^i)}{\mu(B_j^i)}\nu(\tau B_j^i)\\
& \le c_o\nu^s(F_i)+\sum_{j=1}^{\infty}\frac{\mu(A_i\cap B_j^i)}{\mu(B_j^i)}\int_{\tau B_j^i}a\,d\mu.
\end{split}
\end{equation}
Fix $k\in\N$. We can divide the above sum into two parts: let $I_1$ consist of those indices $j\in\N$ for which $\mu(A_i\cap B_j^i)/\mu(B_j^i)>1/k$, and let $I_2$ consist of the remaining indices. We estimate
\[
\mu\left(\bigcup_{j\in I_1}\tau B_j^i\right)\le C\sum_{j\in I_1}\mu(B_j^i)\le Ck\sum_{j\in I_1}\mu(A_i\cap B_j^i)\le Ck\mu(A_i)<\delta,
\]
when $i$ is large enough. Now we can further estimate \eqref{eq:equiintegrability first estimate}:
\[
\int_{A_i}g_i\,d\mu\le c_o\nu^s(F_i)+\frac{c_o}{k}\int_{F_i}a\,d\mu+c_o\eps
\]
for large enough $i\in\N$. By letting first $i\to \infty$, then $k\to \infty$, and finally $\eps\to 0$, we get a contradiction with $\int_{A_i}g_i\,d\mu>\eta>0$, proving the equi-integrability.

Finally, let us prove the weak convergence in $L^1(F)$. Possibly by taking a subsequence which we still denote $g_i$, we have $g_i\to \check{a}$ weakly in $L^1(F)$ for some $\check{a}\in L^1(F)$, by the Dunford-Pettis theorem (see e.g. \cite[Theorem 1.38]{AmbFP00}). By this weak convergence and the bounded overlap property of the Whitney balls, we can estimate for any $x\in F$ and $0<\widetilde{r}<r$
\begin{align*}
\int_{B(x,\widetilde{r})\cap F}\check{a}\,d\mu&=\limsup_{i\to\infty}\int_{B(x,\widetilde{r})\cap F}g_i\,d\mu\\
&=\limsup_{i\to\infty}\sum_{j=1}^{\infty}\frac{\mu(B_j^i\cap B(x,\widetilde{r})\cap F)}{\mu(B_j^i)}\nu(\tau B_j^i)\\
&\le \limsup_{i\to\infty}\sum_{j\in\N:\,B_j^i\cap B(x,\widetilde{r})\cap F\ne \emptyset}\nu(\tau B_j^i)\\
&\le \limsup_{i\to\infty}c_o\nu(B(x,r)).
\end{align*}
By letting $\widetilde{r}\nearrow r$, we get
\[
\int_{B(x,r)\cap F}\check{a}\,d\mu\le c_o\nu(B(x,r)).
\]
By the Radon-Nikodym theorem, $\mu$-almost every $x\in F$ satisfies
\[
\lim_{r\to 0}\,\vint{B(x,r)\cap F}\check{a}\,d\mu=\check{a}(x)\quad\textrm{and}\quad\lim_{r\to 0}\frac{\nu^s(B(x,r))}{\mu(B(x,r))}=0.
\]
By using these estimates as well as the previous one, we get for $\mu$-almost every $x\in F$
\[
\begin{split}
\check{a}(x)&=\lim_{r\to 0}\,\vint{B(x,r)\cap F}\check{a}\,d\mu \\
&\le c_o\limsup_{r\to 0}\,\vint{B(x,r)}a\,d\mu+c_o\limsup_{r\to 0}\frac{\nu^s(B(x,r))}{\mu(B(x,r))},
\end{split}
\]
where the first term on the right-hand side is $c_oa$ by the Radon-Nikodym theorem, and the second term is zero. Thus we have $\check{a}\le c_oa$ $\mu$-almost everywhere in $F$.
\end{proof}

Now we are ready to prove the estimate from above.

\begin{theorem}\label{thm:estimate from above}
Let $\Omega$ be an open set, and let $u\in L^1_{\loc}(\Omega)$ with $\mathcal F(u,\Omega)<\infty$. Let $d\Vert Du\Vert=a\,d\mu+d\,\Vert Du\Vert^s$ be the decomposition of the variation measure, where $a\in L^1(\Omega)$ and $\Vert Du\Vert^s$ is the singular part. Then we have
\[
\mathcal F(u,\Omega)\le \int_{\Omega}f(Ca)\,d\mu+f_{\infty}\Vert Du\Vert^s(\Omega),
\]
with $C=C(c_d,c_P,\lambda)$.
\end{theorem}
\begin{proof}
Since the functional $\mathcal F(u,\cdot)$ is a Radon measure by Theorem \ref{thm:measure_prop}, we can decompose it into the absolutely continuous and singular parts as $\mathcal F(u,\cdot)=\mathcal F^a(u,\cdot)+\mathcal F^s(u,\cdot)$. Again, the singular parts $\Vert Du \Vert^s$ and $\mF^s(u,\cdot)$ are concentrated on a Borel set $D \subset \Omega$ that satisfies $\mu (D)=0$ and 
\[
\Vert Du\Vert^s (\Omega \setminus D)=0=\mF^s(u,\Omega \setminus D),
\] 
see e.g. \cite[p. 42]{EvGa}.

First we prove the estimate for the singular part. Let $\eps >0$. Choose an open set $G$ with $D \subset G \subset \Omega$, such that $\mu (G)<\eps$ and $\Vert Du \Vert (G) < \Vert Du \Vert (D)+\eps$. Take a sequence $u_i \in \liploc(G)$ such that $u_i \to u$ in $L^1_{\loc}(G)$ and
\[
\int_G g_{u_i}\,d\mu \to \Vert Du\Vert (G) \quad\textrm{as}\ \ i\to\infty.
\]
Thus for some $i\in\N$ large enough, we have
\[
\int_G g_{u_i}\,d\mu < \Vert Du\Vert (G)+\eps
\]
and
\[
\mF(u,G)< \int_G f(g_{u_i}) \, d\mu +\eps.
\]
The last inequality necessarily holds for large enough $i$ by the definition of the functional $\mF(u,\cdot)$. Now, using the two inequalities above and the estimate for $f$ given in \eqref{eq:estimate for f}, we can estimate
\begin{align*}
\mF(u,D) & \le \mF(u,G) 
          \le \int_G f(g_{u_i}) \, d\mu +\eps \\
         & \le \int_G f(0)\, d\mu+f_{\infty}\int_G g_{u_i}\, d\mu +\eps \\
         & \le f(0)\mu(G)+f_{\infty}\Vert Du\Vert (G) +f_{\infty}\eps+\eps \\
         & \le f(0)\eps +f_{\infty}( \Vert Du\Vert (D)+\eps) +f_{\infty}\eps+\eps.
\end{align*}
In the last inequality we used the properties of the set $G$ given earlier. Letting $\eps \to 0$, we get the estimate from above for the singular part, i.e.
\begin{equation}\label{eq:estimate from above singular}
\mF^s(u,\Omega)=\mF(u,D)\le f_{\infty}\Vert Du\Vert(D)= f_{\infty}\Vert Du\Vert^s (\Omega).
\end{equation}

Next let us consider the absolutely continuous part.
Let $D$ be defined as above, and let $F=\Omega \setminus D$. Let $\eps >0$. Take an open set $G$ such that $F\subset G\subset \Omega$, and $\Vert Du\Vert (G)<\Vert Du\Vert (F)+\eps$.

For every $i\in\N$, take a Whitney covering $\{B_j^i=B(x_j^i,r_j^i)\}_{j=1}^{\infty}$ of $G$ s.t. $r_j^i\le 1/i$ for every $j\in\N$, $5\lambda B_j^i\subset G$ for every $j\in\N$, every ball $5\lambda B_k^i$ meets at most $C=C(c_d,\lambda)$ balls $5\lambda B_j^i$, and if $5\lambda B_j^i$ meets $5\lambda B_k^i$, then $r_j^i\le 2r_k^i$. Then take a partition of unity $\{\phi_j^i\}_{j=1}^{\infty}$ subordinate to this cover, such that $0\le \phi_j^i\le 1$, each $\phi_j^i$ is a $C(c_d)i$-Lipschitz function, and $\supp(\phi_j^i)\subset 2B_j^i$ for every $j\in\N$ (see e.g. \cite[Theorem 3.4]{BBS07}). Define discrete convolutions with respect to the Whitney coverings by
\[
u_i=\sum_{j=1}^{\infty}u_{B_j^i}\phi_j^i,\quad\ i\in\N.
\]
We know that $u_i\to u$ in $L^1(G)$ as $i\to\infty$, and that each $u_i$ has an upper gradient
\[
g_i=C\sum_{j=1}^{\infty}\chi_{B_j^i}\frac{\Vert Du\Vert(5\lambda B_j^i)}{\mu(B_j^i)}
\]
with $C=C(c_d,c_P)$, see e.g. the proof of \cite[Proposition 4.1]{KKST}. We can of course write the decomposition $g_i=g_i^a+g_i^s$, where
\[
g_i^a=C\sum_{j=1}^{\infty}\chi_{B_j^i}\frac{\int_{5\lambda B_j^i}a\,d\mu}{\mu(B_j^i)}
\]
and
\[
g_i^s=C\sum_{j=1}^{\infty}\chi_{B_j^i}\frac{\Vert Du\Vert^s(5\lambda B_j^i)}{\mu(B_j^i)}.
\]
By the bounded overlap property of the coverings, we can easily estimate
\begin{equation}\label{eq:mass of singular parts}
\int_Gg_i^s\,d\mu\le \widetilde{C}\Vert Du\Vert^s(G)<\widetilde{C}\eps
\end{equation}
for every $i\in\N$, with $\widetilde{C}=\widetilde{C}(c_d,c_P,\lambda)$. Furthermore, by Lemma \ref{lem:generalized equiintegrability} we know that the sequence $g_i^a$ is equi-integrable and that a subsequence, which we still denote $g_i^a$, converges weakly in $L^1(G)$ to a function $\check{a}\le Ca$, with $C=C(c_d,\lambda)$.
By Mazur's lemma we have for certain convex combinations, denoted by a hat,
\[
\widehat{g_i^a}=\sum_{j=i}^{N_i}d_{i,j}g_j^a \to \check{a}\quad\textrm{in}\ L^1(G) \ \textrm{as}\ i\to\infty,
\]
where $d_{i,j}\ge 0$ and $\sum_{j=i}^{N_i}d_{i,j}=1$ for every $i\in \N$ \cite[Theorem 3.12]{Rud}. We note that $\widehat{u_i}\in\liploc(G)$  for every $i\in\N$ (the hat always means that we take the same convex combinations), $\widehat{u_i}\to u$ in $L^1_{\loc}(G)$, and $g_{\widehat{u_i}}\le \widehat{g_i}$ $\mu$-almost everywhere for every $i\in \N$ (recall that $g_u$ always means the minimal $1$-weak upper gradient of $u$).
Using the definition of $\mF(u,\cdot)$, the fact that $f$ is $L$-Lipschitz, and \eqref{eq:mass of singular parts}, we get
\begin{align*}
\mF(u,F)&\le \mF(u,G)
        \le \liminf_{i\to \infty} \int_G f(g_{\widehat{u_i}})\,d\mu\\
        &\le \liminf_{i\to \infty}\int_Gf(\widehat{g_i})\,d\mu
        \le \liminf_{i\to \infty}\left( \int_Gf(\widehat{g_i^a})\,d\mu+\int_GL\widehat{g_i^s}\,d\mu\right)\\
        &\le \liminf_{i\to \infty}\left( \int_Gf(\widehat{g_i^a})\,d\mu+L\widetilde{C}\eps\right)
         =  \int_Gf(\check{a})\,d\mu+L\widetilde{C}\eps\\
        &\le \int_Gf(Ca)\,d\mu+L\widetilde{C}\eps
        \le \int_{\Omega}f(Ca)\,d\mu+L\widetilde{C}\eps.
\end{align*}
By letting $\eps \to 0$ we get the estimate from above for the absolutely continuous part, i.e.
\[
\mF^a(u,\Omega)= \mF(u,F)\le \int_{\Omega}f(Ca)\,d\mu.
\]
By combining this with \eqref{eq:estimate from above singular}, we get the desired estimate from above.
\end{proof}

\begin{remark}\label{rem:integral representation}
By using Theorems \ref{thm:estimate from below} and \ref{thm:estimate from above}, as well as the definition of the functional for general sets given in \eqref{eq:def of F for general sets}, we can conclude that for any $\mu$-measurable set $A\subset\Omega\subset X$ with $\mathcal F(u,\Omega)<\infty$, we have
\[
\mathcal F^s(u,A)=f_{\infty}\Vert Du\Vert^s(A)
\]
and
\[
\int_{A}f(a)\,d\mu\le \mathcal F^a(u,A)\le \int_{A}f(Ca)\,d\mu,
\]
where $\mathcal F^a(u,\cdot)$ and $\mathcal F^s(u,\cdot)$ are again the absolutely continuous and singular parts of the measure given by the functional.
\end{remark}

Since locally Lipschitz functions are dense in the Newtonian space $N^{1,1}(\Omega)$ with $\Omega$ open \cite[Theorem 5.47]{BB}, from the definition of total variation we know that if $u\in N^{1,1}(\Omega)$, then $u\in\BV(\Omega)$ with $\Vert Du\Vert$ absolutely continuous, and more precisely
\[
\Vert Du\Vert(\Omega)\le \int_{\Omega}g_u\,d\mu.
\]
We obtain, to some extent as a by-product of the latter part of the proof of the previous theorem, the following converse, which also answers a question posed in \cite{KKST}. A later example will show that the constant $C$ is necessary here as well.

\begin{theorem}\label{thm:min upper gradient and variation}
Let $\Omega\subset X$ be an open set, let $u\in\BV(\Omega)$, and let $d\Vert Du\Vert=a\,d\mu+d\Vert Du\Vert^s$ be the decomposition of the variation measure, where $a\in L^1(\Omega)$ and $\Vert Du\Vert^s$ is the singular part. Let $F\subset \Omega$ be a $\mu$-measurable set for which $\Vert Du\Vert^s(F)=0$. Then, by modifying $u$ on a set of $\mu$-measure zero if necessary, we have $u|_F\in N^{1,1}(F)$  and $g_u\le Ca$  $\mu$-almost everywhere in $F$, with $C=C(c_d,c_P,\lambda)$.
\end{theorem}

\begin{proof}
We pick a sequence of open sets $F_i$ such that $F\subset F_i\subset \Omega$ and $\Vert Du\Vert^s(F_i)<1/i$, $i=1,2,\ldots$. Then, as described in Lemma \ref{lem:generalized equiintegrability}, we pick Whitney coverings $\{B_j^i\}_{j=1}^{\infty}$ of the sets $F_i$, with the constant $\tau=5\lambda$.

Furthermore, as we did in the latter part of the proof of Theorem \ref{thm:estimate from above} with the open set $G$, we define for every $i\in\N$ a discrete convolution $u_i$ of the function $u$ with respect to the Whitney covering $\{B_j^i\}_{j=1}^{\infty}$. Every $u_i$ has an upper gradient
\[
g_i=C\sum_{j=1}^{\infty}\chi_{B_j^i}\frac{\Vert Du\Vert(5\lambda B_j^i)}{\mu(B_j^i)}
\]
in $F_i$, with $C=C(c_d,c_P)$, and naturally $g_i$ is then also an upper gradient of $u_i$ in $F$. We have $u_i\to u$ in $L^1(F)$ (see e.g. the proof of \cite[Proposition 4.1]{KKST}) and, according to Lemma \ref{lem:generalized equiintegrability} and up to a subsequence, $g_i\to \check{a}$ weakly in $L^1(F)$, where $\check{a}\le Ca$ $\mu$-almost everywhere in $F$. We now know by \cite[Lemma 7.8]{Hj} that by modifying $u$ on a set of $\mu$-measure zero, if necessary, we have that $\check{a}$ is a $1$-weak upper gradient of $u$ in $F$. Thus we have the result.
\end{proof}

\begin{remark}
As in Lemma \ref{lem:generalized equiintegrability}, if $\Vert Du\Vert$ is absolutely continuous on the whole of $\Omega$, we can choose simply $F=\Omega$, and then we also have the inequality
\[
\int_{\Omega}g_u\,d\mu\le C\Vert Du\Vert(\Omega)
\]
with $C=C(c_d,c_P,\lambda)$. Note also that the proof of \cite[Lemma 7.8]{Hj}, which we used above, is also based on Mazur's lemma, so the techniques used above are very similar to those used in the proof of Theorem \ref{thm:estimate from above}.
\end{remark}

Finally we give the counterexample which shows that in general, we can have
\[
\mF^a(u,\Omega)> \int_{\Omega}f(a)\,d\mu
\quad\text{and}\quad\Vert Du\Vert(\Omega)< \int_{\Omega}g_u\,d\mu.
\]
The latter inequality answers a question raised in \cite{M} and later in \cite{AMP}.

\begin{example}
Take the space $X=[0,1]$, equipped with the Euclidean distance and a measure $\mu$, which we will next define. First we construct a fat Cantor set $A$ as follows. Take $A_0=[0,1]$, whose measure we denote $\alpha_0=\mathcal L^1(A_0)=1$, where $\mathcal L^1$ is the 1-dimensional Lebesgue measure. Then in each step $i\in\N$ we remove from $A_{i-1}$ the set $B_i$, which consists of $2^{i-1}$ open intervals of length $2^{-2i}$, centered at the middle points of the intervals that make up $A_{i-1}$. We denote $\alpha_i=\mathcal L^1(A_i)$, and define $A=\cap_{i=1}^{\infty}A_i$. Then we have
\[
\alpha=\mathcal L^1(A) =\lim_{i\to\infty}\alpha_i=1/2.
\]
Now, equip the space $X$ with the weighted Lebesgue measure $d\mu=w\,d\mathcal L^1$, where $w=2$ in $A$ and $w=1$ in $X\setminus A$. Define
\[
g=\frac{1}{\alpha}\chi_{A}=2\chi_{A} \quad\ \textrm{and}\quad\  g_i=\frac{1}{\alpha_{i-1}-\alpha_i}\chi_{B_i},\ \ i\in\N.
\]
The unweighted integral of $g$ and each $g_i$ over $X$ is $1$.
Next define the function
\[
u(x)=\int_0^x g\, d\mathcal L^1.
\]
Now $u$ is in $N^{1,1}(X)$ and even in $\Lip(X)$, since $g$ is bounded. The minimal 1-weak upper gradient of $u$ is $g$ --- this can be seen e.g. by the representation formulas for minimal upper gradients, see \cite[Theorem 2.50]{BB}. Approximate $u$ with the functions
\[
u_i(x)=\int_0^x g_i\, d\mathcal L^1,\quad i\in\N.
\]
The functions $u_i$ are Lipschitz, and they converge to $u$ in $L^1(X)$ and even uniformly. This can be seen as follows. Given $i\in\N$, the set $A_i$ consists of $2^i$ intervals of length $\alpha_i/2^i$. If $I$ is one of these intervals, we have
\[
2^{-i}=\int_Ig\,d\mathcal L^1 =\int_Ig_{i+1}\,d\mathcal L^1,
\]
and also
\[
\int_{X\setminus A_i}g\,d\mathcal L^1 =0=\int_{X\setminus A_i}g_{i+1}\,d\mathcal L^1 .
\]
Hence $u_{i+1}=u$ at the end points of the intervals that make up $A_i$, and elsewhere $|u_{i+1}-u|$ is at most $2^{-i}$.

Clearly the minimal 1-weak upper gradient of $u_i$ is $g_i$. However, we have
\[
\int_0^1 g\,d\mu = 2 > 1 = \lim_{i\to\infty} \int_0^1 g_i\,d\mu \geq \Vert Du\Vert([0,1]).
\]
Thus the total variation is strictly smaller than the integral of the minimal 1-weak upper gradient, demonstrating the necessity of the constant $C$ in Theorem \ref{thm:min upper gradient and variation}. On the other hand, any approximating sequence $u_i\to u$ in $L^1(X)$ converges, up to a subsequence, also pointwise $\mu$- and thus $\mathcal L^1$-almost everywhere, and then we necessarily have for some such sequence
\begin{equation}\label{eq:gradients of approximating sequence}
\Vert Du\Vert([0,1])=\lim_{i\to\infty}\int_0^1g_{u_i}\,d\mu\ge \limsup_{i\to\infty}\int_0^1g_{u_i}\,d\mathcal L^1\ge 1.
\end{equation}
Hence we have $\Vert Du\Vert([0,1])=1$.
Let us show that more precisely, $d\Vert Du\Vert=a\,d\mu$ with $a=\chi_{A}$. The fact that $u$ is Lipschitz implies that $\Vert Du\Vert$ is absolutely continuous with respect to $\mu$. Since $u_i$ converges to $u$ uniformly, for any interval $(d,e)$ we must have
\[
\lim_{i\to\infty} \int_{(d,e)} g_i \,d\mathcal L^1  = \int_{(d,e)} g \,d\mathcal L^1 ,
\]
and since for the weight we had $w=1$ where $g_i> 0$, and $w=2$ where $g> 0$, we now get
\[
\lim_{i\to\infty} \int_{(d,e)} g_i \,d\mu = \frac{1}{2} \int_{(d,e)} g \,d\mu.
\]
By the definition of the variation measure, we have at any point $x\in X$ for $r>0$ small enough
\[
\Vert Du\Vert((x-r,x+r)) \le \liminf_{i\to\infty}\int_{(x-r,x+r)} g_i \,d\mu = \frac{1}{2} \int_{(x-r,x+r)} g \,d\mu.
\]
Now, if $x\in A$, we can estimate the Radon-Nikodym derivative
\[
\limsup_{r\to\infty}\frac{\Vert Du\Vert(B(x,r))}{\mu(B(x,r))}\le 1,
\]
and when $x\in X\setminus A$, we clearly have that the derivative is $0$. On the other hand, if the derivative were strictly smaller than $1$ in a subset of $A$ of positive $\mu$-measure, we would get $\Vert Du\Vert(X)<1$, which is a contradiction with the fact that $\Vert Du\Vert(X)=1$. 
Thus $d\Vert Du\Vert=a\,d\mu$ with $a=\chi_{A}$.
\footnote{We can further show that $g_i\,d\mu \overset{*}{\rightharpoonup} a\,d\mu$ in $X$, but we do not have $g_i\to a$ weakly in $L^1(X)$, demonstrating the subtle difference between the two types of weak convergence.}

To show that we can have $\mathcal F^a(u,X)> \int_{X}f(a)\,d\mu$ --- note that $\mathcal F^a(u,X)=\mathcal F(u,X)$ ---  assume that $f$ is given by
\[
f(t)=\begin{cases}
t,     & t\in[0,1], \\
2t-1,  & t>1.
\end{cases}
\]
(We could equally well consider other nonlinear $f$ that satisfy the earlier assumptions.)
Since $a=\chi_{A}$, we have
\[
\int_{X}f(a)\,d\mu=\int_{X}a\,d\mu= 2\int_X\chi_A\,d\mathcal L^1=1.
\]
On the other hand, for some sequence of Lipschitz functions $v_i\to u$ in $L^1(X)$, we have
\begin{equation}\label{eq:example}
\begin{split}
\mF(u,X) &=\lim_{i\to\infty}\int_{X}f(g_{v_i})\,d\mu\\
         &=\lim_{i\to\infty} \left(2\int_Af(g_{v_i})\,d\mathcal L^1+\int_{X\setminus A}f(g_{v_i})\,d\mathcal L^1 \right).
\end{split}
\end{equation}
By considering a subsequence, if necessary, we may assume that $v_i\to u$ pointwise $\mu$- and thus $\mathcal L^1$-almost everywhere.
By Proposition \ref{prop:weak convergence}, we have for any closed set $F\subset X\setminus A$
\[
\limsup_{i\to\infty}\int_{F} f(g_{v_i})\,d\mu\le\mF(u,F)\le\mF(u,X\setminus A)\le\int_{X\setminus A}f(g_u)\,d\mu=0,
\]
which implies that
\[
\lim_{i\to\infty}\int_{F} f(g_{v_i})\,d\mathcal L^1=0=\lim_{i\to\infty}\int_{F} g_{v_i}\,d\mathcal L^1.
\]
Applying these two equalities together with the inequality $f(t)\geq 2t-1$, we obtain
\begin{align*}
\limsup_{i\to\infty} \int_{X\setminus A} f(g_{v_i}) \,d\mathcal L^1 &= \limsup_{i\to\infty} \int_{X\setminus (A\cup F)} f(g_{v_i}) \,d\mathcal L^1\\
&\ge \limsup_{i\to\infty} \int_{X\setminus (A\cup F)} (2g_{v_i}-1) \,d\mathcal L^1\\
&\ge \limsup_{i\to\infty} \int_{X\setminus (A\cup F)} 2g_{v_i} \,d\mathcal L^1 - \mathcal L^1(X\setminus (A\cup F))\\
&=\limsup_{i\to\infty} \int_{X\setminus A} 2 g_{v_i} \,d\mathcal L^1 - \mathcal L^1(X\setminus (A\cup F)).
\end{align*}
The last term on the last line can be made arbitrarily small.
Inserting this into \eqref{eq:example}, we get
\begin{align*}
\mF(u,X)&=\limsup_{i\to\infty} \left(2\int_A f(g_{v_i})\,d\mathcal L^1+\int_{X\setminus A}f(g_{v_i})\,d\mathcal L^1 \right)\\
&\ge 2\liminf_{i\to\infty}\int_Af(g_{v_i})\,d\mathcal L^1+ 2\limsup_{i\to\infty}\int_{X\setminus A}g_{v_i}\,d\mathcal L^1 \\
&\ge 2\liminf_{i\to \infty}\int_0^1g_{v_i}\,d\mathcal L^1\ge 2.
\end{align*}
The last inequality follows from the pointwise convergence of $v_i$ to $u$ $\mathcal L^1$-almost everywhere.

Roughly speaking, we note that the total variation $\Vert Du\Vert(X)$ is found to be unexpectedly small because the growth of the approximating functions $u_i$ is concentrated outside the Cantor set $A$, where it is ``cheaper'' due to the smaller value of the weight function. However, when we calculate $\mF(u,X)$, the same does not work, because now the nonlinear function $f$ places ``extra weight'' on upper gradients that take values larger than $1$.
\end{example}

\section{Minimization problem}

Let us consider a minimization problem related to the functional of linear growth. First we specify what we mean by boundary values of $\BV$ functions. 

\begin{definition}
Let $\Omega$ and $\Omega^*$ be bounded open subsets of $X$ such that $\Omega\Subset\Omega^*$, and assume that $h\in\BV(\Omega^*)$. 
We define $\BV_{h}(\Omega)$ as the
space of functions $u\in\BV(\Omega^*)$ such that $u=h$ $\mu$-almost everywhere in $\Omega^*\setminus\Omega$.
\end{definition}

Now we give the definition of our minimization problem.
\begin{definition}\label{def:minimization problem}
A function $u\in \BV_{h}(\Omega)$  is a minimizer of the functional of linear growth
with the boundary values $h\in\BV(\Omega^*)$, if
\[
\mathcal F(u,\Omega^*)= \inf\mathcal F(v,\Omega^*),
\]
where the infimum is taken over all $v\in\BV_h(\Omega)$.
\end{definition}

Note that if $u\in L^1_{\loc}(\Omega^*)$ and $u=h$ in $\Omega^*\setminus \Omega$, then $u\in L^1(\Omega^*)$. Furthermore, if $\mF(u,\Omega^*)<\infty$, then $\Vert Du\Vert(\Omega^*)<\infty$ by \eqref{eq:basic estimate for functional}. Thus it makes sense to restrict $u$ to the class $\BV(\Omega^*)$ in the above definition.
Observe that the minimizers do not depend on $\Omega^*$, but the value of the functional does. 
Note also that the minimization problem always has a solution and that the solution is not necessarily continuous, see \cite{HKL}.

\begin{remark}
We point out that any minimizer is also a local minimizer in the following sense. A minimizer $u\in \BV_{h}(\Omega)$ of $\mathcal F(\cdot,\Omega^*)$ with the boundary values $h\in\BV(\Omega^*)$ is
a minimizer of $\mathcal F(\cdot,\Omega'')$ with the  boundary values $u\in\BV_{u}(\Omega')$ for every $\Omega'\Subset\Omega''\subset\Omega^*$, with $\Omega'\subset\Omega$. This can be seen as follows. Every
$v\in\BV_{u}(\Omega')$ can be extended to $\Omega^*$ by defining $v=u$ in $\Omega^*\setminus\Omega''$. The minimality of $u$ and the measure property of the functional (Theorem \ref{thm:measure_prop}) then imply that
\[
\mF(u,\Omega^*\setminus\Omega'') + \mF(u,\Omega'')\leq\mF(v,\Omega^*\setminus\Omega'') + \mF(v,\Omega''). 
\]
Since $u=v$ $\mu$-almost everywhere in $\Omega^*\setminus\Omega'$, the first terms on both sides of the inequality cancel out, and we have
\[
\mF(u,\Omega'')\leq \mF(v,\Omega''). 
\]
\end{remark}

Now we wish to express the boundary values of the minimization problem as a penalty term involving an integral over the boundary. 
To this end, we need to discuss boundary traces and extensions of $\BV$ functions.

\begin{definition}
An open set $\Omega$ is a strong $\BV$ extension domain, if for every $u\in \BV(\Omega)$ there is an extension $Eu\in \BV(X)$ such that $Eu|_{\Omega}=u$, there is a constant $1\le c_{\Omega}<\infty$ such that $\Vert Eu\Vert_{\BV(X)}\le c_{\Omega}\Vert u\Vert_{\BV(\Omega)}$ and $\Vert D(Eu)\Vert(\partial\Omega)=0$.
\end{definition}
Note that our definition differs from the conventional definition of a $\BV$ extension domain, since we also require that $\Vert D(Eu)\Vert(\partial\Omega)=0$. This can be understood as an additional regularity condition for the domain. 

\begin{definition}
We say that a $\mu$-measurable set $\Omega$ satisfies the \emph{weak measure density condition} if for $\mathcal H$-almost every $x\in\partial \Omega$, we have
\[
\liminf_{r\to 0}\frac{\mu(B(x,r)\cap\Omega)}{\mu(B(x,r))}>0.
\]
\end{definition}

These are the two conditions we will impose in order to have satisfactory results on the boundary traces of $\BV$ functions.
Based on results found in \cite{BS}, we prove in the upcoming note \cite{L} that every bounded \emph{uniform} domain is a strong $\BV$ extension domain and satisfies the weak measure density condition. An open set $\Omega$ is $A$-uniform, with constant $A\ge 1$, if for every $x,y\in\Omega$ there is a curve $\gamma$ in $\Omega$ connecting $x$ and $y$ such that $\ell_{\gamma}\le Ad(x,y)$, and for all $t\in [0,\ell_{\gamma}]$, we have
\[
\dist(\gamma(t),X\setminus\Omega)\ge A^{-1}\min\{t,\ell_{\gamma}-t\}.
\]


Now we give the definition of boundary traces.
\begin{definition}
For a $\mu$-measurable set $\Omega$ and a $\mu$-measurable function $u$ on $\Omega$, a real-valued function $T_{\Omega}u$ defined on $\partial\Omega$ is a boundary trace of $u$ if for $\mathcal H$-almost every $x\in\partial\Omega$, we have
\[
\lim_{r\to 0}\,\vint{\Omega\cap B(x,r)}|u-T_{\Omega}u(x)|\,d\mu=0.
\]
\end{definition}

Often we will also call $T_{\Omega}u(x)$ a boundary trace if the above condition is satisfied at the point $x$.
If the trace exists at a point $x\in\partial\Omega$, we clearly have
\[
T_{\Omega}u(x)=\lim_{r\to 0}\,\vint{B(x,r)\cap\Omega}u\,d\mu=\aplim\limits_{y\in\Omega,\, y\to x}u(y),
\]
where $\aplim$ denotes the approximate limit.
Furthermore, we can show that the trace is always a Borel function.

Let us recall the following decomposition result for the variation measure of a $\BV$ function from \cite[Theorem 5.3]{AMP}. For any open set $\Omega\subset X$, any $u\in\BV(\Omega)$, and any Borel set $A\subset \Omega$ that is $\sigma$-finite with respect to $\mathcal H$, we have
\begin{equation}\label{eq:decomposition}
\Vert Du\Vert(\Omega)=\Vert Du\Vert(\Omega\setminus A)+\int_A\int_{u^{\wedge}(x)}^{u^{\vee}(x)}\theta_{\{u>t\}}(x)\,dt\,d\mathcal H(x).
\end{equation}
The function $\theta$ and the lower and upper approximate limits $u^{\wedge}$ and $u^{\vee}$ were defined in Section \ref{sec:prelis}. In particular, by \cite[Theorem 5.3]{AMP} the jump set $S_u$ is known to be $\sigma$-finite with respect to $\mathcal H$.

The following is our main result on boundary traces.

\begin{theorem}\label{thm:boundary traces}
Assume that $\Omega$ is a strong $\BV$ extension domain that satisfies the weak measure density condition, and let $u\in\BV(\Omega)$. Then the boundary trace $T_{\Omega}u$ exists, that is, $T_{\Omega}u(x)$ is defined for $\mathcal H$-almost every $x\in\partial\Omega$.
\end{theorem}
\begin{proof}
Extend $u$ to a function $Eu\in\BV(X)$. By the fact that 
\[
\Vert D(Eu)\Vert(\partial \Omega)=0
\] 
and the decomposition \eqref{eq:decomposition}, we have $\mathcal H(S_{Eu}\cap \partial\Omega)=0$ --- recall that the function $\theta$ is bounded away from zero. Here 
\[
S_{Eu}=\{x\in X:\,(Eu)^{\wedge}(x)<(Eu)^{\vee}(x)\},
\] 
as usual. On the other hand, by \cite[Theorem 3.5]{KKST} we know that $\mathcal H$-almost every point $x\in\partial^*\Omega\setminus S_{Eu}$ is a Lebesgue point of $Eu$. In these points we define $T_{\Omega}u(x)$ simply as the Lebesgue limit $\widetilde{Eu}(x)$. For $\mathcal H$-almost every $x\in\partial\Omega$ the weak measure density condition is also satisfied, so that
\[
\liminf_{r\to 0}\frac{\mu(B(x,r)\cap\Omega)}{\mu(B(x,r))}=c>0.
\]
Thus for $\mathcal H$-almost every $x\in\partial\Omega$ we can estimate
\[
\begin{split}
\limsup_{r\to 0}\,&\vint{B(x,r)\cap\Omega}|u-T_{\Omega}u(x)|\,d\mu\\
&\le \limsup_{r\to 0}\frac{1}{c\mu(B(x,r))}\int_{B(x,r)}|Eu-\widetilde{Eu}(x)|\,d\mu=0.
\end{split}
\]
\end{proof}

Due to the Lebesgue point theorem \cite[Theorem 3.5]{KKST}, we have in fact
\[
\limsup_{r\to 0}\,\vint{B(x,r)\cap\Omega}|u-T_{\Omega}u(x)|^{Q/(Q-1)}\,d\mu=0
\]
for $\mathcal H$-almost every $x\in\partial\Omega$, where $Q>1$ was given in \eqref{eq:doubling dimension}. However, we will not need this stronger result.

Let us list some general properties of boundary traces.

\begin{proposition}\label{prop:prop of trace}
Assume that $\Omega$ is a $\mu$-measurable set and that $u$ and $v$ are $\mu$-measurable functions on $\Omega$. The boundary trace operator enjoys the following properties for any $x\in\partial\Omega$ for which both $T_{\Omega}u(x)$ and $T_{\Omega}v(x)$ exist: 
\begin{enumerate}[(i)]
\item $T_{\Omega}(\alpha u + \beta v)(x)= \alpha\, T_{\Omega}u(x) +\beta\, T_{\Omega}v(x)$ for any $\alpha,\beta\in\R$.\\
\item If $u\ge v$ $\mu$-almost everywhere~in $\Omega$, then $T_{\Omega} u(x)\ge T_{\Omega} v(x)$. In particular, if
$u=v$ $\mu$-almost everywhere~in $\Omega$, then  $T_{\Omega} u(x)= T_{\Omega} v(x)$.\\
\item $T_{\Omega}(\max\{u,v\})(x)=\max\{T_{\Omega}u(x),T_{\Omega}v(x)\}$ and
$T_{\Omega}(\min\{u,v\})(x)=\min\{T_{\Omega}u(x),T_{\Omega}v(x)\}$.\\
\item Let $h>0$ and define the truncation $u_h=\min\{h,\max\{u,-h\}\}$. Then $T_{\Omega} u_h (x)=(T_{\Omega}u(x))_h$.\\
\item If $\Omega$ is a $\mu$-measurable set such that both $\Omega$ and its complement satisfy the weak measure density condition, and $w$ is a $\mu$-measurable function on $X$, then for $\mathcal H$-almost everywhere $x\in\partial\Omega$ for which both traces $T_{\Omega}w(x)$ and $T_{X\setminus \Omega}w(x)$ exist, we have
\[
\{T_{\Omega} w(x), T_{X\setminus\Omega}w(x)\}=\{w^{\wedge}(x), w^{\vee}(x)\}.
\]
\end{enumerate}
\end{proposition}
\begin{proof}
Assertions $(i)$ and $(ii)$ are clear. Since minimum and maximum can be written as sums by using absolute values, property $(iii)$ follows from $(i)$ and the easily verified fact that $T_{\Omega}|u|(x)=|T_{\Omega}u(x)|$. Assertion $(iv)$ follows from $(iii)$. In proving assertion $(v)$, due to the symmetry of the situation we can assume that $T_{\Omega}w(x)\ge T_{X\setminus\Omega}w(x)$. By using the definition of traces and Chebyshev's inequality, we deduce that for every $\eps>0$,
\[
\lim_{r\to 0}\frac{\mu(\{|w-T_{\Omega}w(x)|>\eps\}\cap B(x,r)\cap\Omega)}{\mu(B(x,r)\cap\Omega)}=0
\]
and
\[
\lim_{r\to 0}\frac{\mu(\{|w-T_{X\setminus\Omega}w(x)|>\eps\}\cap B(x,r)\setminus\Omega)}{\mu(B(x,r)\setminus\Omega)}=0.
\]
To determine the lower and upper approximate limits, we use these results to compute
\begin{align*}
&\limsup_{r\to 0}\frac{\mu(\{w>t\}\cap B(x,r))}{\mu(B(x,r))}\\
& =\limsup_{r\to 0}\left[\frac{\mu(\{w>t\}\cap B(x,r)\cap\Omega)}{\mu(B(x,r))}+\frac{\mu(\{w>t\}\cap B(x,r)\setminus\Omega)}{\mu(B(x,r))} \right]\\
&\begin{cases}
=0+0,                                                &\textrm{if }t>T_{\Omega}w(x), \\
=\limsup_{r\to 0}\frac{\mu(B(x,r)\cap\Omega)}{\mu(B(x,r))}+0,        &\textrm{if }T_{X\setminus\Omega}w(x)<t<T_{\Omega}w(x), \\
= \limsup_{r\to 0}\left[\frac{\mu(B(x,r)\cap\Omega)}{\mu(B(x,r))}+\frac{\mu(B(x,r)\setminus\Omega)}{\mu(B(x,r))}\right],  &\textrm{if }t<T_{X\setminus\Omega}w(x),
\end{cases}\\
&\begin{cases}
=0,                                                &\textrm{if }t>T_{\Omega}w(x), \\
\in (0,1),        &\textrm{if }T_{X\setminus\Omega}w(x)<t<T_{\Omega}w(x), \\
 =1,  &\textrm{if }t<T_{X\setminus\Omega}w(x).
\end{cases}
\end{align*}
To obtain the result ``$\in (0,1)$'' above, we used the weak measure density conditions.
We conclude that $w^{\vee}(x)=T_{\Omega}w(x)$, and since ``$\limsup$'' can be replaced by ``$\liminf$'' in the above calculation, we also get $w^{\wedge}(x)=T_{X\setminus\Omega}w(x)$.
\end{proof}

A minor point to be noted is that any function that is in the class $\BV(X)$, such as an extension $Eu$ for $u\in \BV(\Omega)$, is also in the class $\BV(\Omega)$, and thus $T_{\Omega}Eu=T_{\Omega}u$.

Eventually we will also need to make an additional assumption on the space, as described in the following definition which is from \cite[Definition 6.1]{AMP}. The function $\theta_E$ was introduced earlier in (\ref{eq:def of theta}).

\begin{definition}
We say that $X$ is a \emph{local space} if, given any two sets of locally finite perimeter $E_1\subset E_2\subset X$, we have $\theta_{E_1}(x)=\theta_{E_2}(x)$ for $\mathcal H$-almost every $x\in \partial^*E_1\cap\partial^*E_2$.
\end{definition}
For some examples of local spaces, see \cite{AMP} and the upcoming note \cite{L}.
The assumption $E_1\subset E_2$ can, in fact, be removed as follows. Note that for a set of locally finite perimeter $E$, we have $\Vert D\chi_E\Vert=\Vert D\chi_{X\setminus E}\Vert$, i.e. the two measures are equal \cite[Proposition 4.7]{M}. From this it follows that $\theta_E(x)=\theta_{X\setminus E}(x)$ for $\mathcal H$-almost every $x\in\partial^*E$. Now, if $E_1$ and $E_2$ are arbitrary sets of locally finite perimeter, we know that $E_1\cap E_2$ and $E_1\setminus E_2$ are also sets of locally finite perimeter \cite[Proposition 4.7]{M}. For every $x\in\partial^*E_1\cap\partial^*E_2$ we have either $x\in\partial^*(E_1\cap E_2)$ or $x\in\partial^*(E_1\setminus E_2)$. Thus by the locality condition, we have for $\mathcal H$-almost every $x\in\partial^*E_1\cap\partial^*E_2$ either
\[
\theta_{E_1}(x)=\theta_{E_1\cap E_2}(x)=\theta_{E_2}(x)
\]
or
\[
\theta_{E_1}(x)=\theta_{E_1\setminus E_2}(x)=\theta_{X\setminus E_2}(x)=\theta_{E_2}(x).
\]
Thus we have $\theta_{E_1}=\theta_{E_2}$ for $\mathcal H$-almost every $x\in\partial^*E_1\cap\partial^*E_2$.

In a local space the decomposition \eqref{eq:decomposition} takes a simpler form, as proved in the following lemma.

\begin{lemma}\label{lem:consequence of locality}
If $X$ is a local space, $\Omega$ is a set of locally finite perimeter, $u\in\BV(X)$, and $A\subset \partial^*\Omega$ is a Borel set, then we have
\[
\int_A\int_{u^{\wedge}(x)}^{u^{\vee}(x)}\theta_{\{u>t\}}(x)\,dt\,d\mathcal H(x)=\int_A(u^{\vee}(x)-u^{\wedge}(x))\theta_{\Omega}\,d\mathcal H(x).
\]
\end{lemma}
Note that since $\Omega$ is a set of locally finite perimeter, $A\subset\partial^*\Omega$ is $\sigma$-finite with respect to $\mathcal H$.
\begin{proof}
We have
\begin{align*}
 &\int_A\int_{u^{\wedge}(x)}^{u^{\vee}(x)}\theta_{\{u>t\}}(x)\,dt\,d\mathcal H(x)\\
 &=\int_A\int_{-\infty}^{\infty}\chi_{\{(u^{\wedge}(x),u^{\vee}(x))\}}(t)\theta_{\{u>t\}}(x)\,dt\,d\mathcal H(x)\\
 & =\int_{-\infty}^{\infty}\int_A\chi_{\{(-\infty,t)\}}(u^{\wedge}(x))\chi_{\{(t,\infty)\}}(u^{\vee}(x))\theta_{\{u>t\}}(x)\,d\mathcal H(x)\,dt\\
 & =\int_{-\infty}^{\infty}\int_{A\cap\partial^*\{u>t\}}\chi_{\{(-\infty,t)\}}(u^{\wedge}(x))\chi_{\{(t,\infty)\}}(u^{\vee}(x))\theta_{\{u>t\}}(x)\,d\mathcal H(x)\,dt.
 \end{align*}
On the third line we used Fubini's theorem. On the fourth line we used the fact that if $u^{\wedge}(x)<t<u^{\vee}(x)$, then $x\in\partial^*\{u>t\}$. 
This follows from the definitions of the lower and upper approximate limits. 
By the locality condition we see that the right-hand side above equals to
 \begin{align*}
 &\int_{-\infty}^{\infty}\int_{A\cap\partial^*\{u>t\}}\chi_{\{(-\infty,t)\}}(u^{\wedge}(x))\chi_{\{(t,\infty)\}}(u^{\vee}(x))\theta_{\Omega}(x)\,d\mathcal H(x)\,dt\\
 & = \int_{-\infty}^{\infty}\int_{A}\chi_{\{(-\infty,t)\}}(u^{\wedge}(x))\chi_{\{(t,\infty)\}}(u^{\vee}(x))\theta_{\Omega}(x)\,d\mathcal H(x)\,dt\\
 &=\int_A\int_{-\infty}^{\infty}\chi_{\{(u^{\wedge}(x),u^{\vee}(x))\}}(t)\,dt\,\theta_{\Omega}(x)\,d\mathcal H(x)\\
 &=\int_A(u^{\vee}(x)-u^{\wedge}(x))\theta_{\Omega}(x)\,d\mathcal H(x).
\end{align*}
\end{proof}

Now we prove two propositions concerning boundary traces that are based on \cite[Theorem 3.84]{AmbFP00} and \cite[Theorem 3.86]{AmbFP00}.
\begin{proposition}\label{prop:gluing}
Let $\Omega$ and $\Omega^*$ be open sets such that $\Omega$ and $\Omega^*\setminus \Omega$ satisfy the weak measure density condition, $\overline{\Omega}\subset \Omega^*$, and $\Omega$ is of finite perimeter.
Let $u,v\in \BV(\Omega^*)$, and let $w=u\chi_{\Omega}+v\chi_{\Omega^*\setminus \Omega}$. Then $w\in \BV(\Omega^*)$ if and only if
\begin{equation}\label{eq:trace integrability}
\int_{\partial \Omega}|T_{\Omega}u-T_{\Omega^*\setminus \overline{\Omega}}\,v|\,d\mathcal H<\infty.
\end{equation}
In the above characterization, we implicitly assume that the integral is well-defined --- in particular, this is the case if $\Omega$ and $\Omega^*\setminus\overline{\Omega}$ are also strong $\BV$ extension domains, due to Theorem \ref{thm:boundary traces}.
Furthermore, if $X$ is a local space, we then have
\[
\Vert Dw\Vert(\Omega^*)= \Vert Du\Vert(\Omega)+\Vert Dv\Vert(\Omega^*\setminus\overline{\Omega})+\int_{\partial \Omega}|T_{\Omega}u-T_{\Omega^*\setminus \overline{\Omega}}\,v|\theta_{\Omega}\,d\mathcal H.
\]
\end{proposition}
\begin{proof}
First note that by the weak measure density conditions, we have $\mathcal H(\partial\Omega\setminus\partial^*\Omega)=0$, and thus $\mathcal H(\partial\Omega)<\infty$. This further implies that $\mu(\partial\Omega)=0$ \cite[Lemma 6.1]{KKST12}, and by this and the weak measure density conditions again, 
\[
\mathcal H(\partial\Omega\setminus\partial\overline{\Omega})=0
\quad\text{and}\quad T_{\Omega^*\setminus \overline{\Omega}}=T_{\Omega^*\setminus \Omega}.
\]

To prove one direction, let us assume \eqref{eq:trace integrability}. In particular, we assume that $T_{\Omega}u(x)$ and $T_{\Omega^*\setminus \overline{\Omega}}\,v(x)$ exist for $\mathcal H$-almost every $x\in\partial\Omega$. For $h>0$, define the truncated functions
\[
u_h=\min\{h,\max\{u,-h\}\}
\qquad\text{and}\qquad v_h=\min\{h,\max\{v,-h\}\}.
\]
Clearly $u_h,v_h,\chi_{\Omega},\chi_{\Omega^*\setminus \Omega}\in\BV(\Omega^*)\cap L^{\infty}(\Omega^*)$. Then
\[
w_h=u_h\chi_{\Omega}+v_h\chi_{\Omega^*\setminus \Omega}\in\BV(\Omega^*)\cap L^{\infty}(\Omega^*),
\]
see e.g. \cite[Proposition 4.2]{KKST}.
Based on the decomposition of the variation measure given in \eqref{eq:decomposition},
\begin{equation}\label{eq:gluing estimate}
\begin{split}
&\Vert Dw_h\Vert (\Omega^*)\\
&=\Vert Du_h\Vert(\Omega)+\Vert Dv_h\Vert(\Omega^*\setminus\overline{\Omega})+\int_{\partial \Omega}\int_{w_h^{\wedge}(x)}^{w_h^{\vee}(x)}\theta_{\{w_h>t\}}(x)\,dt\,d\mathcal H(x)\\
                       &\le \Vert Du\Vert(\Omega)+\Vert Dv\Vert(\Omega^*\setminus\overline{\Omega})+\int_{\partial \Omega}c_d|w_h^{\vee}(x)-w_h^{\wedge}(x)|\,d\mathcal H(x).
\end{split}
\end{equation}
By Proposition \ref{prop:prop of trace} $(iv)$, the boundary traces $T_{\Omega}$ of $u$, $u_h$, $w_h$, and $T_{\Omega^*\setminus\overline{\Omega}}$ of $v$, $v_h$, $w_h$, exist $\mathcal H$-almost everywhere on the boundary $\partial \Omega$. For $w_h$ this fact follows from the definition of boundary traces, by which we have that $T_{\Omega}w_h=T_{\Omega}u_h$, and similarly $T_{\Omega^*\setminus\overline{\Omega}}\,w_h=T_{\Omega^*\setminus\overline{\Omega}}\,v_h$. Proposition \ref{prop:prop of trace} $(v)$ now gives
\begin{equation}\label{eq:gluing traces}
\left\{w_h^{\wedge}(x), w_h^{\vee}(x)\right\} =\{T_{\Omega}w_h(x), T_{\Omega^*\setminus\overline{\Omega}}\,w_h(x)\}=\{T_{\Omega}u_h(x), T_{\Omega^*\setminus\overline{\Omega}}\,v_h(x)\}
\end{equation}
for $\mathcal H$-almost every $x\in \partial \Omega$.
Using Proposition \ref{prop:prop of trace} $(iv)$ again, for $\mathcal H$-almost every~$x\in\partial\Omega$ we have
\begin{equation}\label{eq:truncated traces}
\begin{split}
&T_{\Omega}u_h(x)=\min\{h,\max\{T_{\Omega}u(x),-h\}\},\\
&T_{\Omega^*\setminus\overline{\Omega}}\,v_h(x)=\min\{h,\max\{T_{\Omega^*\setminus\overline{\Omega}}\,v(x),-h\}\}.
\end{split}
\end{equation}
By the lower semicontinuity of the total variation as well as \eqref{eq:gluing estimate}, \eqref{eq:gluing traces} and \eqref{eq:truncated traces}, we now get
\begin{align*}
\Vert &Dw\Vert(\Omega^*)\le \liminf_{h\to \infty}\Vert Dw_h\Vert(\Omega^*)\\
                      &\le \Vert Du\Vert(\Omega)+\Vert Dv\Vert(\Omega^*\setminus\overline{\Omega})+\liminf_{h\to \infty}c_d\int_{\partial \Omega}|T_{\Omega}u_h-T_{\Omega^*\setminus\overline{\Omega}}\,v_h| \,d\mathcal H\\
                     &= \Vert Du\Vert(\Omega)+\Vert Dv\Vert(\Omega^*\setminus\overline{\Omega})+c_d\int_{\partial \Omega}|T_{\Omega}u-T_{\Omega^*\setminus\overline{\Omega}}\,v|\,d\mathcal H
                     <\infty.
\end{align*}
Thus $w\in\BV(\Omega^*)$.

To prove the converse, assume that $w\in \BV(\Omega^*)$. Here we can simply again write the decomposition of the variation measure
\[
\infty >\Vert Dw\Vert(\Omega^*)\ge \Vert Du\Vert(\Omega)+\Vert Dv\Vert(\Omega^*\setminus\overline{\Omega})+\alpha\int_{\partial \Omega}|w^{\vee}-w^{\wedge}|\,d\mathcal H,
\]
where $\alpha=\alpha(c_d,c_P)>0$, and just as earlier, note that
\begin{equation}\label{eq:approximate limits and traces}
|w^{\vee}(x)-w^{\wedge}(x)|=|T_{\Omega}w(x)-T_{\Omega^*\setminus \overline{\Omega}}\,w(x)|=|T_{\Omega}u(x)-T_{\Omega^*\setminus \overline{\Omega}}\,v(x)|
\end{equation}
for $\mathcal H$-almost every $x\in \partial \Omega$. This combined with the previous estimate gives the desired result. If $X$ is a local space, we combine the decomposition of the variation measure \eqref{eq:decomposition}, Lemma \ref{lem:consequence of locality}, and \eqref{eq:approximate limits and traces} to obtain the last claim.
\end{proof}

Next we show that if a set $A$ (which could be e.g. the boundary $\partial \Omega$) is in a suitable sense of codimension one, traces of $\BV$ functions are indeed integrable on $A$.
Let us first recall the following fact from the theory of sets of finite perimeter. Given any set of finite perimeter $E\subset X$, for $\mathcal H$-almost every $x\in \partial^*E$ we have
\begin{equation}\label{eq:density of E}
\gamma \le \liminf_{r\to 0} \frac{\mu(E\cap B(x,r))}{\mu(B(x,r))} \le \limsup_{r\to 0} \frac{\mu(E\cap B(x,r))}{\mu(B(x,r))}\le 1-\gamma,
\end{equation}
where $\gamma \in (0,1/2]$ only depends on the doubling constant and the constants in the Poincar\'e inequality \cite[Theorem 5.4]{A2}.

\begin{proposition}\label{prop:codimension one boundary}
Let $\Omega^*\subset X$ be open, let $u\in \BV(\Omega^*)$, and let $A\subset \Omega^*$ be a bounded Borel set that satisfies $\dist(A,X\setminus \Omega^*)>0$ and
\begin{equation}\label{eq:codimension one condition}
\mathcal H(A\cap B(x,r))\le c_A\frac{\mu(B(x,r))}{r}
\end{equation}
for every $x\in A$ and $r\in (0,R]$, where $R\in(0,\dist(A,X\setminus \Omega^*))$ and $c_A>0$ are constants. Then
\begin{equation}\label{eq:summability of traces}
\int_{A}(|u^{\wedge}|+|u^{\vee}|)\,d\mathcal{H}
 \le C\Vert u\Vert_{\BV(\Omega^*)},
\end{equation}
where $C= C(c_d,c_P,\lambda,A,R,c_A)$.
\end{proposition}
\begin{proof}
We may assume that $u\ge 0$. Let
\[
c=\inf_{x\in A}\mu(B(x,R));
\]
by the doubling property of $\mu$ we have $c=c(A,R,c_d)>0$. First consider a set $E\subset X$ that is of finite perimeter in $\Omega^*$ and satisfies $\mu(E)<\delta$, where $\delta>0$ is a constant that will be determined later. Define
\[
E^{\gamma}=\left\{x\in \Omega^*:\,\liminf_{r\to 0}\frac{\mu(E\cap B(x,r))}{\mu(B(x,r))}\ge \gamma\right\},
\]
where $\gamma=\gamma(c_d,c_P,\lambda)>0$ is the constant from \eqref{eq:density of E}. Pick any $x\in E^{\gamma}\cap A$. We note that
\[
\frac{\mu(E\cap B(x,R))}{\mu(B(x,R))}\le \frac{\mu(E)}{\mu(B(x,R))}< \frac{\delta}{c}.
\]
By choosing $\delta>0$ small enough, we have
\[
\frac{\mu(E\cap B(x,R/(5\lambda)))}{\mu(B(x,R/(5\lambda)))}\le \frac{\gamma}{2}.
\]
Thus we have $\delta=\delta(c_d,\lambda,c,\gamma)$, and consequently $\delta=\delta(c_d,c_P,\lambda,A,R)$. By the definition of $E^{\gamma}$, we can find a number $r\in (0,R/5]$ that satisfies
\[
\frac{\gamma}{2c_d}<\frac{\mu(E\cap B(x,r/\lambda))}{\mu(B(x,r/\lambda))}\le \frac{\gamma}{2}.
\]
This can be done by repeatedly halving the radius $R/5$ until the right-hand side of the above inequality does not hold, and picking the last radius for which it did hold. From the relative isoperimetric inequality \eqref{eq:isop ineq} we conclude that
\begin{equation}\label{eq:estimate for balls}
\frac{\mu(B(x,r/\lambda))}{r/\lambda}\le \frac{2c_d}{\gamma}\frac{\mu(E\cap B(x,r/\lambda))}{r/\lambda}\le \frac{C}{\gamma} P(E,B(x,r)).
\end{equation}
Using the radii chosen this way, we get a covering $\{B(x,r(x))\}_{x\in A\cap E^{\gamma}}$ of the set $A\cap E^{\gamma}$. By the 5-covering lemma, we can select a countable family of disjoint balls $\{B(x_i,r_i)\}_{i=1}^{\infty}$ such that the balls $B(x_i,5r_i)$ cover $A\cap E^{\gamma}$. By using \eqref{eq:codimension one condition} and \eqref{eq:estimate for balls}, we get
\begin{equation}\label{eq:estimate for E gamma}
\begin{split}
\mathcal H(E^{\gamma}\cap A) &\le \sum_{i=1}^{\infty}\mathcal H(E^{\gamma}\cap A\cap B(x_i,5r_i))\\
                                             &\le c_A\sum_{i=1}^{\infty}\frac{\mu(B(x_i,5r_i))}{5r_i}
                                             \le C\sum_{i=1}^{\infty}\frac{\mu(B(x_i,r_i/\lambda))}{r_i/\lambda}\\
                                             &\le C\sum_{i=1}^{\infty}P(E,B(x_i,r_i))
                                             \le CP(E,\Omega^*),
\end{split}
\end{equation}
where $C=(c_d,c_P,\lambda,c_A)$.

Then we consider the function $u$. Assume that $x\in A\cap S_u$ and $u^{\wedge}(x)+u^{\vee}(x)>t$, with $t>0$. By the definitions of the lower and upper approximate limits, we know that $x\in \partial^{*}\{u>s\}$ for all $s\in (u^{\wedge}(x),u^{\vee}(x))$. By the coarea formula \eqref{eq:coarea}, the sets $\{u>s\}$ are of finite perimeter in $\Omega^*$ for every $s\in T$, where $T$ is a countable dense subset of $\R$. Thus, outside a $\mathcal H$-negligible set, \eqref{eq:density of E} holds for every $x\in\partial^{*}\{u>s\}$ and $s\in T$. Assuming that $x$ is outside this $\mathcal H$-negligible set, we can find $s\in ((u^{\wedge}(x)+u^{\vee}(x))/2,u^{\vee}(x))\cap T$ and estimate
\[
\liminf_{r\to 0}\frac{\mu(\{u>t/2\}\cap B(x,r))}{\mu(B(x,r))}\ge \liminf_{r\to 0}\frac{\mu(\{u>s\}\cap B(x,r))}{\mu(B(x,r))}\ge \gamma,
\]
which means that $x\in \{u>t/2\}^{\gamma}$. By Chebyshev's inequality we get
\[
\mu(\{u>t/2\})\le \frac{\Vert u\Vert_{L^1(\Omega^*)}}{t/2}<\delta
\]
if $t>t_0$, where $t_0=C(c_d,c_P,\lambda,A,R)\Vert u\Vert_{L^1(\Omega^*)}$ due to the dependencies of $\delta$ given earlier.
By the coarea formula again, $\{u>t/2\}$ is of finite perimeter in $\Omega^*$ for a.e. $t\in\R$, and Cavalieri's principle and \eqref{eq:estimate for E gamma} then imply that
\begin{align*}
\int_{A\cap S_u}&(u^{\wedge}+u^{\vee})\,d\mathcal H =\int_{0}^{\infty}\mathcal H(\{x\in A\cap S_u:u^{\wedge}(x)+u^{\vee}(x)>t\})\,dt\\
                                              &\le\int_{0}^{\infty}\mathcal H(\{u>t/2\}^{\gamma}\cap A)\,dt\\
                                              &\le t_0\mathcal H(A)+\int_{t_0}^{\infty}C(c_d,c_P,\lambda,c_A)P(\{u>t/2\},\Omega^*)\,dt\\
                                              &\le C(c_d,c_P,\lambda,A,R)\Vert u\Vert_{L^1(\Omega^*)}\mathcal H(A)+C(c_d,c_P,\lambda,c_A)\Vert Du\Vert(\Omega^*).
\end{align*}
This gives the estimate for $A\cap S_u$. For $A\setminus S_u$, we simply note that if $x\in A\setminus S_u$ and $u^{\wedge}(x)=u^{\vee}(x)>t$, then the approximate limit of $u$ at $x$ is larger than $t$, which easily gives $x\in\{u>t\}^{\gamma}$, and then we can use Cavalieri's principle as above.
\end{proof}

Finally we get the desired representation for the minimization problem.

\begin{theorem}
Assume that $X$ is a local space, and let $\Omega\Subset \Omega^*$ be bounded open sets such that $\Omega$ and $\Omega^*\setminus\Omega$ satisfy the weak measure density condition, $\Omega$ is a strong $\BV$ extension domain, and $\partial \Omega$ satisfies the assumptions of Proposition \ref{prop:codimension one boundary}. Assume also that $h\in\BV(\Omega^*)$ and that the trace $T_{X\setminus\overline{\Omega}}h(x)$ exists for $\mathcal H$-almost every $x\in\partial\Omega$, which in particular is true if $\Omega^*\setminus\overline{\Omega}$ is also a strong $\BV$ extension domain. Then the minimization problem given in Definition \ref{def:minimization problem}, with boundary values $h$, can be reformulated as the minimization of the functional
\begin{equation}\label{eq:reformulation}
\mathcal F(u,\Omega)+f_{\infty}\int_{\partial \Omega}|T_{\Omega}u-T_{X\setminus \overline{\Omega}}h|\theta_{\Omega}\,d\mathcal H
\end{equation}
over all $u\in \BV(\Omega)$.
\end{theorem}
Note that this formulation contains no reference to $\Omega^*$.
\begin{proof}
First note that due to the conditions of Proposition \ref{prop:codimension one boundary}, we have $\mathcal H(\partial \Omega)<\infty$, and thus $\mu(\partial\Omega)=0$ and $\Omega$ is a set of finite perimeter, see e.g. \cite[Lemma 6.1, Proposition 6.3]{KKST12}.
By the weak measure density conditions, 
\[
\mathcal H(\partial\Omega\setminus\partial\overline{\Omega})=0
\quad\text{and}\quad T_{\Omega^*\setminus \overline{\Omega}}=T_{\Omega^*\setminus \Omega}.
\]
Now, for any $u\in \BV_h(\Omega)$, we have $u\in\BV(\Omega^*)$ by definition, and $\mathcal F(u,\Omega^*)<\infty$ by \eqref{eq:basic estimate for functional}. Then
\begin{equation}\label{eq:reformulation calculation}
\begin{split}
\mF&(u,\Omega^*)\\
&= \mF(u,\Omega)+\mF^s(u,\partial \Omega)+\mF(h,\Omega^*\setminus \overline{\Omega})\\
               &= \mF(u,\Omega)+f_{\infty}\Vert Du\Vert^s(\partial \Omega)+\mF(h,\Omega^*\setminus \overline{\Omega})\\
               &= \mF(u,\Omega)+f_{\infty}\int_{\partial \Omega}|u^{\vee}-u^{\wedge}|\theta_{\Omega}\,d\mathcal H+\mF(h,\Omega^*\setminus \overline{\Omega})\\
               &= \mF(u,\Omega)+f_{\infty}\int_{\partial \Omega}|T_{\Omega}u-T_{X\setminus \overline{\Omega}}h|\theta_{\Omega}\,d\mathcal H+\mF(h,\Omega^*\setminus \overline{\Omega}),
\end{split}
\end{equation}
where the first equality follows from the measure property of $\mathcal F(u,\cdot)$ as well as the fact that $\mu(\partial \Omega)=0$, the second equality follows from the integral representation of the functional (see Remark \ref{rem:integral representation}), the third equality follows from the decomposition \eqref{eq:decomposition} and Lemma \ref{lem:consequence of locality}, and the fourth equality follows from Proposition \ref{prop:prop of trace} $(v)$. Now, the term $\mF(h,\Omega^*\setminus \overline{\Omega})$ does not depend on $u$, so in fact we need to minimize \eqref{eq:reformulation}.

Conversely, assume that $u\in \BV(\Omega)$. Then we can extend $u$ to $Eu\in\BV(\Omega^*)$. By Proposition \ref{prop:prop of trace} $(v)$ we have 
\[
\{T_{\Omega}h(x),T_{X\setminus\overline{\Omega}}\,h(x)\}=\{h^{\wedge}(x),h^{\vee}(x)\}
\] 
for $\mathcal H$-almost every $x\in\partial\Omega$. By the proof of Theorem \ref{thm:boundary traces} we have that $T_{\Omega}Eu(x)$ is the Lebesgue limit of $Eu$ for $\mathcal H$-almost every $x\in\partial\Omega$. By Proposition \ref{prop:codimension one boundary}, we now get
\[
\int_{\partial \Omega}|T_{\Omega}Eu-T_{X\setminus \overline{\Omega}}h|\,d\mathcal H\le C(\Vert Eu\Vert_{\BV(\Omega^*)}+\Vert h\Vert_{\BV(\Omega^*)})<\infty.
\]
By Proposition \ref{prop:gluing} we deduce that $w=(Eu)\chi_{\Omega}+h\chi_{\Omega^*\setminus\Omega}\in\BV(\Omega^*)$, and in fact we have $w=u\chi_{\Omega}+h\chi_{\Omega^*\setminus\Omega}\in\BV_{h}(\Omega)$. This completes the proof.
\end{proof}

\begin{remark}
Note that in the latter part of the above proof we showed that, under the assumptions on the space and on $\Omega$, the spaces $\BV(\Omega)$ and $\BV_h(\Omega)\subset \BV(\Omega^*)$ can be identified.
\end{remark}

\noindent Addresses:\\

\noindent H.H.:  Department of Mathematical Sciences, P.O. Box 3000, FI-90014 University of Oulu, Finland. \\
\noindent E-mail: {\tt heikki.hakkarainen@oulu.fi}\\

\smallskip
\noindent J.K.: Aalto University, School of Science, Department of Mathematics, P.O. Box 11100, FI-00076 Aalto, Finland. \\
\noindent E-mail:  {\tt juha.k.kinnunen@aalto.fi}\\

\smallskip
\noindent P.L.: Aalto University, School of Science, Department of Mathematics, P.O. Box 11100, FI-00076 Aalto, Finland. \\
\noindent E-mail: {\tt panu.lahti@aalto.fi}\\

\smallskip
\noindent P.L.: Aalto University, School of Science, Department of Mathematics, P.O. Box 11100, FI-00076 Aalto, Finland. \\
\noindent E-mail: {\tt pekka.lehtela@aalto.fi}\\

\end{document}